\documentclass[12pt]{amsart}

\usepackage{color,amssymb,amsmath, tikz-cd, amscd,latexsym, epsfig, amsthm}
\usepackage{cleveref}
\usepackage{lipsum}
\newtheorem{theorem}{Theorem}[section]
\newtheorem{proposition}[theorem]{Proposition}
\newtheorem{lemma}[theorem]{Lemma}
\newtheorem{corollary}[theorem]{Corollary}

\theoremstyle{definition}

\newtheorem{conjecture}[theorem]{Conjecture}
\newtheorem{remark}[theorem]{Remark}

\newcommand{\NN}{ \ensuremath{\mathbb{N}}}
\newcommand{\ZZ}{ \ensuremath{\mathbb{Z}}}

\newcommand{\RR}{ \ensuremath{\mathbb{R}}}
\newcommand{\QQ}{ \ensuremath{\mathbb{Q}}}

\newcommand{\aaa}{{\mathbf{a}}}
\newcommand{\bb}{{\mathbf{b}}}
\newcommand{\cc}{{\mathbf{c}}}
\newcommand{\ttt}{\mathbf{t}}
\newcommand{\ee}{{\mathbf{e}}}

\newcommand{\Hilb}{\mathrm{Hilb}}
\newcommand{\lk}{{\mathrm{lk}}}
\newcommand{\st}{\mathrm{st}}

\newcommand{\mideal}{\ensuremath{\mathfrak{m}}}

\def\cocoa{{\hbox{\rm C\kern-.13em o\kern-.07em C\kern-.13em o\kern-.15em A}}}

\newcommand{\FF}{\mathbb{F}}
\newcommand{\hav}{\overline h}
\newcommand{\bdelta}{\mbox{\boldmath $\delta$}}
\newcommand{\supp}{{s}}
\newcommand{\eenu}{\mathfrak e}
\newcommand{\cano}[1]{\Omega ({#1})}

\begin{document}

\title[Generalized Lower Bound Theorem]{A generalized lower bound theorem\\ for balanced manifolds}

\author[M. Juhnke-Kubitzke]{Martina Juhnke-Kubitzke}
\address{
Martina Juhnke-Kubitzke, 
FB 6 -- Institut f\"ur Mathematik,
Universit\"at Osnabr\"uck,
Albrechtstr. 28a, 49076 Osnabr\"uck, Germany
}
\email{juhnke-kubitzke@uni-osnabrueck.de}

\author[S. Murai]{Satoshi Murai}
\address{
Satoshi Murai,
Department of Pure and Applied Mathematics,
Graduate School of Information Science and Technology,
Osaka University,
Suita, Osaka, 565-0871, Japan}
\email{s-murai@ist.osaka-u.ac.jp}

\author[I. Novik]{Isabella Novik}
\address{
Isabella Novik, 
Department of Mathematics, 
University of Washington, 
Seattle, WA 98195-4350, USA}
\email{novik@math.washington.edu}

\author[C. Sawaske]{Connor Sawaske}
\address{Connor Sawaske,
Department of Mathematics, 
University of Washington, 
Seattle, WA 98195-4350, USA}
\email{sawaske@math.washington.edu}

\thanks{Juhnke-Kubitzke's research is partially supported by German Research Council DFG-GRK 1916. Murai's research is partially supported by JSPS KAKENHI JP16K05102. Novik's research is partially supported by NSF grant DMS-1361423}


\begin{abstract}
A simplicial complex of dimension $d-1$ is said to be balanced if its graph is $d$-colorable. Juhnke-Kubitzke and Murai proved an analogue of the generalized lower bound theorem for balanced simplicial polytopes. We establish a generalization of their result to balanced triangulations of closed homology manifolds and balanced triangulations of orientable homology manifolds with boundary under an additional assumption that all proper links of these triangulations have the weak Lefschetz property. As a corollary, we show that if $\Delta$ is an arbitrary  balanced triangulation of any closed homology manifold of dimension $d-1 \geq 3$, then  $2h_2(\Delta) - (d-1)h_1(\Delta)  \geq 4{d \choose 2}(\tilde{\beta}_1(\Delta)-\tilde{\beta}_0(\Delta))$, thus verifying a conjecture by Klee and Novik. To prove these results we develop the theory of flag $h''$-vectors.
\end{abstract}

\maketitle

\section{Introduction}\label{intro}
At the intersection of geometry, algebra, and combinatorics is the study of the face numbers of simplicial complexes. If $f_i(\Delta)$ denotes the number of $i$-dimensional faces of a $(d-1)$-dimensional simplicial complex $\Delta$, then the $h$-numbers $h_i(\Delta)$ are defined by $h_i(\Delta)=\sum_{j=0}^i(-1)^{i-j}{d-i\choose j-i}f_{j-1}(\Delta)$. Of the most important results in the study of face numbers of simplicial complexes, many have been elegantly phrased in the language of the $h$-numbers. Principal among these are the Dehn--Sommerville relations, the lower and upper bound theorems, and their culmination -- the $g$-theorem. Our starting point is the following generalized lower bound theorem (or GLBT) conjectured by McMullen and Walkup \cite{McMW} and proved by Stanley \cite{St-80}, and Murai and Nevo \cite{MN-13}:
\begin{theorem} \label{GLBT}Let $P$ be a $d$-dimensional simplicial polytope. Then
\[
h_0(P)\le h_1(P)\le\cdots\le h_{\lfloor\frac{d}{2}\rfloor}(P);
\]
also the equality $h_{i-1}(P)=h_i(P)$ occurs for a certain $i\le \lfloor{\frac{d}{2}}\rfloor$ if and only if $P$ is $(i-1)$-stacked.
\end{theorem}

It is natural to ask to what extent these inequalities can be specialized. In particular, are there classes of simplicial polytopes whose successive $h$-numbers satisfy more drastic inequalities? Of recent interest have been balanced simplicial complexes (those complexes whose underlying graphs have a ``minimal'' coloring), introduced by Stanley in \cite{St79}. Examples of balanced simplicial complexes include barycentric subdivisions of regular CW complexes, Coxeter complexes, and Tits buildings. The following strengthening of Theorem \ref{GLBT} for balanced simplicial polytopes was conjectured in \cite{KN} and proved by Juhnke-Kubitzke and Murai in \cite{JM}.

\begin{theorem} \label{balanced_GLBT_conjecture}Let $P$ be a $d$-dimensional  balanced simplicial polytope. Then
\[
\frac{h_0(P)}{{d \choose 0}} \le \frac{h_1(P)}{{d \choose 1}} \le \cdots \le \frac{h_i(P)}{{d \choose i}} \le \cdots \le \frac{h_{\lfloor d/2\rfloor}(P)}{{d \choose  \lfloor d/2\rfloor}}.
\]
\end{theorem}

Our goal is to examine extensions of this result to more general complexes. In particular, the complexes considered in this paper are balanced $\FF$-homology manifolds with and without boundary, where $\FF$ is a field. (We defer most of the definitions until the following sections.) When confining our attention to this class of simplicial complexes, the natural analog of the $h$-numbers turns out to be the $h''$-numbers (for polytopes, these are one and the same): for a $(d-1)$-dimensional  complex $\Delta$ and $i<d$,  $h_i''(\Delta)$ is defined by $h_i(\Delta)-{d\choose i}\sum_{j=1}^i(-1)^{i-j}\tilde{\beta}_{j-1}(\Delta)$, where $\tilde{\beta}_{j-1}(\Delta)$, $1\leq j\leq d$, are the reduced Betti numbers computed over $\FF$. Specifically, the manifold GLBT asserts that if $\Delta$ is a $(d-1)$-dimensional $\FF$-homology manifold  with or without boundary whose vertex links have the weak Lefschetz property, then $h''_i(\Delta, \partial\Delta)\geq h''_{i-1}(\Delta, \partial\Delta)+{d \choose i-1}\tilde{\beta}_{i-1}(\Delta, \partial\Delta)$ for all $i\leq \lfloor d/2\rfloor$; see \cite[eq.~(9)]{NS3} and \cite[Theorem 1.5]{MN-bdry}.
In view of this result, it seems plausible that the statement of Theorem \ref{balanced_GLBT_conjecture} can be appropriately extended to balanced $\FF$-homology manifolds. Indeed, the following is one of our main results.

\begin{theorem} \label{BGLBT}
Let $\Delta$ be a $(d-1)$-dimensional balanced $\FF$-homology manifold with or without boundary and let $1\le \ell\leq \lfloor d/2\rfloor$ be an integer. Assume further that the link of each codimension-$(2\ell-1)$ face of $\Delta$ has the weak Lefschetz property. 
\begin{enumerate}
\item[(i)] If $\Delta$ has no boundary, then
$ \displaystyle
\frac{h''_{\ell}(\Delta)}{{d \choose \ell}} \geq \frac{h''_{\ell-1}(\Delta)}{{d \choose \ell-1}} + \tilde{\beta}_{\ell-1}(\Delta).
$
\item[(ii)] If $\Delta$ is an \textbf{orientable} $\FF$-homology manifold with non-empty boundary, then
$ \displaystyle
\frac{h''_{\ell}(\Delta, \partial\Delta)}{{d \choose \ell}} \geq \frac{h''_{\ell-1}(\Delta, \partial\Delta)}{{d \choose \ell-1}} + \tilde{\beta}_{\ell-1}(\Delta, \partial\Delta).
$
\end{enumerate}
\end{theorem}

Recall that by \cite{St-80},  the boundary complexes of all simplicial polytopes even have the strong Lefschetz property over $\QQ$. Thus Theorem \ref{BGLBT}(i) holds for all balanced triangulations of $\QQ$-homology manifolds with polytopal vertex links and all $\ell$. Moreover, according to 
\cite[Corollary 3.5]{Mu-10} and \cite{Whit},
triangulations of $2$-spheres have the weak Lefschetz property over \emph{any} field $\FF$. Hence, the case $\ell=2$ of Theorem \ref{BGLBT}(i) is valid for \emph{any} balanced $\FF$-homology manifold without boundary.
We prove the following stronger result.

\begin{theorem} \label{g2''}
Let $\Delta$ be a $(d-1)$-dimensional balanced simplicial complex. 
If $\Delta$ is an  $\FF$-homology manifold without boundary and $d\ge 4$, then
\[
\frac{h_2''(\Delta)}{{d \choose 2}}\ge \frac{h_1''(\Delta)}{{d \choose 1}}+\tilde{\beta}_1(\Delta).
\]
Equivalently, $2h_2(\Delta)-(d-1)h_1(\Delta) \ge 4{d \choose 2}(\tilde{\beta}_1(\Delta)-\tilde{\beta}_0(\Delta))$. 
Furthermore, if $d\geq 5$, then this inequality holds as equality if and only if each connected component of $\Delta$ is in the balanced Walkup class. 
\end{theorem}
  
	This result provides a balanced analog of \cite[Theorem 5.2]{NS1} (see also \cite[Theorem 5.3]{Mu}) and settles Conjecture 4.14 of \cite{KN} (see also \cite[Remark 3.8]{KN}).
It is worth mentioning that for $d-1\geq 4$, the condition that $\Delta$ is in the balanced Walkup class is equivalent to all vertex links of $\Delta$ being stacked cross-polytopal spheres (see \cite[Corollary 4.12]{KN}). 
 
  

We also extend Theorem \ref{g2''} to the class of Buchsbaum* simplicial complexes introduced by Athanasiadis and Welker \cite{AW-12} as well as discuss extensions of Theorem \ref{BGLBT}(i) to this generality, under an additional assumption that proper links of the Buchsbaum* complex in question satisfy a certain conjecture of Bj\"orner and Swartz.

Our proofs combine techniques from \cite{JM} along with recent results on Buchsbaum complexes, most notably those from \cite{MNY}.
In particular, we extend the exploitation of $\NN^m$-gradings (rather than the usual $\NN$-grading) to (certain quotients of) the canonical modules 
of Stanley--Reisner rings of balanced Buchsbaum complexes. For most of the proofs
we need to work in the generality of {\em $\aaa$-balanced} simplicial complexes with $\aaa \in \NN^m$. (As $m$ varies, this class of complexes interpolates between the class of balanced simplicial complexes and that of all simplicial complexes.) We introduce the notions of {\em flag $h'$-} and {\em  flag $h''$-vectors} for $\aaa$-balanced simplicial complexes as  flag analogs of the usual $h'$- and $h''$-vectors, and develop basic properties of these vectors from the viewpoint of the Stanley--Reisner ring theory. 

The layout of the rest of the paper is as follows. In Section 2 we recall several notions pertaining to balanced simplicial complexes and their Stanley--Reisner rings.
In Section 3 we introduce flag $h'$- and $h''$-vectors and develop their basic properties. 
Section 4 is devoted to the proof of both parts of Theorem \ref{BGLBT} for the case of orientable homology manifolds with and without boundary. 
In Section 5, we review some known results on canonical modules as well as develop new techniques for studying Stanley--Reisner rings via the canonical modules of the links.
In Section 6 we provide a proof of Theorem \ref{BGLBT}(i), and hence also of the inequality part of Theorem \ref{g2''} for all (closed) homology manifolds. Section 7 settles the equality part of Theorem \ref{g2''}.  
We finish with some remarks and open problems in Section 8.

Initially, the main result of this paper was proved by the team of Juhnke-Kubitzke and Murai, and by the team of Novik and Sawaske. We decided to combine our efforts in a joint paper.




\section{Algebraic properties and combinatorics of simplicial complexes}\label{sect:prel}
Here we review several notions and results that are used in the rest of the paper.

\subsection{Combinatorics of simplicial complexes}\label{sect:Combinatorics}
We start with several definitions. An excellent reference to this material is Stanley's book \cite{StGr}.
Let $V$ be a finite set. An (abstract) \textbf{simplicial complex} $\Delta$ on the vertex set $V$ is a collection of subsets of $V$ that is closed under inclusion and contains all singletons $\{v \}$ with $v \in V$.
Throughout this paper, we assume that all simplicial complexes are finite.
Elements of $\Delta$ are called \textbf{faces} of $\Delta$
and maximal faces (with respect to inclusion) are called \textbf{facets} of $\Delta$.
The \textbf{dimension} of a face $\sigma \in \Delta$ is its cardinality minus one, and the \textbf{dimension} of $\Delta$ is the maximal dimension of its faces.
The $0$-dimensional faces are called \textbf{vertices}, and we denote by $V(\Delta)$ the set of vertices of $\Delta$.
 We say that a simplicial complex $\Delta$ is \textbf{pure} if all facets of $\Delta$ have the same dimension. 

If $\dim \Delta=d-1$, then the \textbf{$f$-vector} of $\Delta$ is $f(\Delta)=(f_{-1}(\Delta), f_0(\Delta), \ldots, f_{d-1}(\Delta))$, where $f_i(\Delta)$ denotes the number of $i$-dimensional faces of $\Delta$, and the \textbf{$h$-vector} of $\Delta$ is $h(\Delta)=(h_0(\Delta), h_1(\Delta), \ldots, h_d(\Delta))$, where $h_i(\Delta)$ is defined by
\[
h_i(\Delta)=\sum_{j=0}^i(-1)^{i-j}{d-j\choose i-j}f_{j-1}(\Delta).
\]
When $P$ is a $d$-dimensional simplicial polytope, $f(P)$ and $h(P)$ refer to the $f$-vector and the $h$-vector of the boundary complex of $P$, respectively.

Given a fixed field $\FF$, denote by $\tilde{\beta_i}(\Delta)=\dim_\FF \tilde{H}_i(\Delta;\FF)$ the $i$\textsuperscript{th} reduced Betti number of $\Delta$ computed over $\FF$. We define the \textbf{$h''$-numbers} of $\Delta$ by
\[
h_i''(\Delta)=
\left\{
\begin{array}{ll} 
h_i(\Delta)-{d\choose i}\sum_{j=1}^i(-1)^{i-j}\tilde{\beta}_{j-1}(\Delta), & \mbox{ for $0\le i\le d-1$,} \\
\tilde{\beta}_{d-1}(\Delta), & \mbox{ for $i=d$}.
\end{array}
\right.
\] 
The Betti numbers and the $h''$-numbers depend on $\FF$, but $\FF$ is usually understood from the context and is omitted from our notation.

A $(d-1$)-dimensional simplicial complex $\Delta$ is called \textbf{balanced} if its underlying graph is $d$-colorable, that is, there exists a map $\pi:V(\Delta)\to [d]=\{1,\ldots,d\}$ such that $\pi(v)\not=\pi(w)$ if $\{v, w\}\in \Delta$. As an example, consider the $d$-dimensional cross-polytope, i.e., the convex hull of the set $\{ \ee_1, \ldots, \ee_d, -\ee_1, \ldots, -\ee_d\}$, where $\{\ee_1,\ldots,\ee_d\}$ is the standard basis of $\RR^d$. Assigning vertices $\ee_i$ and $-\ee_i$ color $i$ for all $1\leq i\leq d$, makes the boundary complex of this polytope into a balanced sphere, denoted $C^*_d$. 

As a generalization of balanced simplicial complexes, we now recall the definition of $\aaa$-balanced simplicial complexes. 
Let $\NN$ denote the set of non-negative integers, and as above let $\ee_1,\dots,\ee_m$ denote the standard basis for $\ZZ^m$. 
For $\bb =(b_1,\dots,b_m) \in \NN^m$,
let $|\bb| = b_1+ \cdots +b_m$.
When $\bb=(b_1,\dots,b_m)$, $\cc=(c_1,\dots,c_m) \in \NN^m$,
we say that $\bb \leq \cc$ if $b_i \leq c_i$ for all $i$; in such a case, we define 
$$ { \cc \choose \bb} := \prod_{i=1}^m {c_i \choose b_i}.$$
Let $\aaa=(a_1,\dots,a_m) \in \NN^m$.
An \textbf{$\aaa$-balanced} simplicial complex is a tuple $(\Delta,\pi)$, where 
\begin{itemize}
\item[(i)] $\Delta$ is a simplicial complex of dimension $|\aaa|-1$; and
\item[(ii)] $\pi$ is a map from $V(\Delta)$ to $\{\ee_1,\dots,\ee_m\}$
such that for every face $\sigma \in \Delta$, $\pi(\sigma)=\sum_{v \in \sigma} \pi(v) \leq \aaa$.
\end{itemize}
For convenience, we also say that $\Delta$ is \textbf{$\aaa$-balanced} if $(\Delta,\pi)$ is $\aaa$-balanced for some $\pi$. In this paper, $\pi$ will often be referred to as a \textbf{coloring} of $\Delta$.  
Note that a $(d-1)$-dimensional simplicial complex $\Delta$ is $(1,1,\dots,1)$-balanced, if and only if its $1$-skeleton is $d$-colorable, which happens if and only if $\Delta$ is balanced.
Also, any $(d-1)$-dimensional simplicial complex can be seen as a monochromatic balanced simplicial complex
(that is, a $(d)$-balanced simplicial complex).

For an $\aaa$-balanced simplicial complex $(\Delta,\pi)$ and $\bb\in \NN^m$,
we denote by  $f_\bb(\Delta,\pi)$  the number of faces $\sigma \in \Delta$ with $\pi(\sigma)=\bb$,
and we define
$$h_\bb (\Delta,\pi) = \sum_{\cc \leq \bb} (-1)^{|\bb|-|\cc|} { \aaa - \cc \choose \bb - \cc} f_{\cc}(\Delta,\pi).$$
The vectors $(f_\bb(\Delta,\pi): \bb \leq \aaa)$ and $(h_\bb(\Delta,\pi): \bb \leq \aaa)$ are called the \textbf{flag $f$-vector} and the \textbf{flag $h$-vector} of $(\Delta,\pi)$, respectively.
These vectors refine the usual $f$- and $h$-vectors, as it is easily seen that
$f_{i-1}(\Delta)=\sum_{\bb \leq \aaa,\ |\bb |=i} f_\bb(\Delta,\pi)$ and $h_i(\Delta)=\sum_{\bb\leq \aaa,\ |\bb|=i} h_\bb(\Delta,\pi)$
for all $0\leq i\leq d$.

\subsection{Stanley--Reisner rings of balanced simplicial complexes}
In this subsection, we recall some basic properties of Stanley--Reisner rings of $\aaa$-balanced simplicial complexes, originally proved by Stanley in \cite{St79}.
In the following, let $\FF$ be an infinite field and let $\Delta$ be a simplicial complex on vertex set $V=V(\Delta)$.
Let $A$ be the polynomial ring $\FF[x_v: v \in V]$ and let $\mideal=(x_v: v \in V)$ be the graded maximal ideal of $A$.
For $\sigma \subseteq V$, we write $x_\sigma= \prod_{v \in \sigma} x_v$.

The \textbf{Stanley--Reisner ideal} $I_\Delta$ of $\Delta$ is the ideal of $A$ defined by
\[I_\Delta=(x_\sigma ~:~ \sigma \subseteq V,\ \sigma \not \in \Delta).\] 
The \textbf{Stanley--Reisner ring} $\FF[\Delta]$ of $\Delta$ (over $\FF$) is the quotient ring
$$\FF[\Delta]=A/I_\Delta.$$ 
If $(\Delta,\pi)$ is an $\aaa$-balanced simplicial complex (where $\aaa\in \NN^m$), the rings $A$ and $\FF[\Delta]$ have the
following $\NN^m$-graded structure induced by the coloring $\pi$: 
$$\deg x_v = \pi(v) \in \NN^m \quad \mbox{for} \quad v\in V.$$

For an $\NN^m$-graded $A$-module $M$ and $\bb\in \NN^m$, we denote by $M_\bb$ the submodule of $M$ consisting of all homogeneous elements of degree $\bb$, and we write $M(-\bb)$ for the module $M$ with the grading defined by $M(-\bb)_{\aaa}=M_{\aaa-\bb}$, where $\aaa\in \NN^m$.  
We will also make use of the submodules
\[M_{\geq\aaa}:=\bigoplus_{\bb\geq \aaa} M_\bb.
\]

The ($\NN^m$-graded) \textbf{Hilbert series} of $M$ is the formal power series in variables $t_1,\dots,t_m$ defined by
$$\Hilb(M;t_1,\dots,t_m):= \sum_{ \bb \in \NN^m} (\dim_\FF M_\bb) \ttt^\bb, \quad \mbox{where $\ttt^\bb = t_1^{b_1} \cdots t_m^{b_m}$.} $$

\begin{theorem}[{Stanley \cite[Section 3]{St79}}]
\label{2.1}
If $(\Delta,\pi)$ is an $\aaa$-balanced simplicial complex, then
$$ \Hilb(\FF[\Delta];t_1,\dots,t_m)= \frac {\sum_{\bb \leq \aaa} h_\bb(\Delta,\pi) \ttt^\bb } {(1-t_1)^{a_1} \cdots (1-t_m)^{a_m}}.$$
\end{theorem}

For a finitely generated graded $A$-module $M$ of Krull dimension $d$,
a \textbf{homogeneous system of parameters} for $M$ is a sequence $\Theta=\theta_1,\dots,\theta_d$ of $d$ homogeneous elements in $\mideal$ such that $\dim_\FF M/\Theta M < \infty$.
Such a system is called a \textbf{linear system of parameters} (or l.s.o.p.\ for short) if it consists of linear forms. By the Noether normalization lemma, if $\FF$ is infinite, then an l.s.o.p.\ always exists. 
A system of parameters $\Theta$ for $M$ is \textbf{$\NN^m$-graded} if each $\theta_i$ is homogeneous w.r.t.\ the $\NN^m$-grading of $A$.
In this case, each $\theta_i$ is a linear combination of variables $x_v$ of the same color (i.e., $\pi(v)=\pi(w)$ for any $x_v$ and $x_w$ that occur in $\theta_i$ with non-zero coefficients). 

\begin{theorem}[{Stanley \cite[Theorem 4.1]{St79}}]
\label{2.2}
Let $(\Delta,\pi)$ be an $\aaa$-balanced simplicial complex.
Then $\FF[\Delta]$ admits an $\NN^m$-graded l.s.o.p. $\Theta$.
Moreover, $(\FF[\Delta]/\Theta \FF[\Delta])_\bb=0$ for any $\bb \in \NN^m$ with $ \bb \not \leq \aaa$.
\end{theorem}

We note that if $\Theta=\theta_1,\dots,\theta_{|\aaa|}$ is an $\NN^m$-graded l.s.o.p.\ for the Stanley--Reisner ring of an $\aaa$-balanced simplicial  complex, then $\Theta$ contains exactly $a_i$ linear forms of degree $\ee_i$ for each $i$ (this follows, for instance, from \cite[Lemma III.2.4]{StGr}).

\subsection{Buchsbaum and Cohen--Macaulay simplicial complexes}

For a simplicial complex $\Delta$ and a face $\sigma \in \Delta$, the simplicial complexes
$$\st_\Delta(\sigma)= \{ \tau \in \Delta~: ~\tau \cup \sigma \in \Delta\} \; \mbox{ and } \;
\lk_\Delta(\sigma)=\{ \tau \in \Delta~:~ \tau \cup \sigma \in \Delta,\ \tau \cap \sigma=\emptyset\}$$
are called the \textbf{star} and the \textbf{link} of $\sigma$ in $\Delta$, respectively.
We say that the link of $\sigma$, $\lk_\Delta(\sigma)$, is \textbf{proper} if $\sigma \ne \emptyset$.
If $(\Delta,\pi)$ is a pure $\aaa$-balanced simplicial complex, then so is $(\st_\Delta(\sigma),\pi)$; furthermore, 
$(\lk_\Delta(\sigma),\pi)$ is an $(\aaa-\pi(\sigma))$-balanced simplicial complex. (Here, $\pi$ is identified with its restriction to the vertex sets of $\st_\Delta(\sigma)$ and $\lk_\Delta(\sigma)$, respectively.)

Recall that a finitely generated graded $A$-module $M$ of Krull dimension $d$ is \textbf{Buchsbaum} if for every homogeneous system of parameters $\Theta=\theta_1,\dots,\theta_d$ of $M$, 
$$ (\theta_1,\dots,\theta_{i-1})M:_M \theta_i= (\theta_1,\dots,\theta_{i-1})M:_M \mideal \quad \mbox{for all }1\leq i\leq d.$$
If, additionally, the above colon module is zero for all $1\leq i \leq d$, then $M$ is said to be \textbf{Cohen--Macaulay}.

We call a simplicial complex $\Delta$ \textbf{Buchsbaum} or \textbf{Cohen--Macaulay} (over $\FF$) if $\FF[\Delta]$ is Buchsbaum or Cohen--Macaulay considered as an $A$-module. 
It is known that a simplicial complex $\Delta$ of dimension $d-1$ is Cohen--Macaulay over $\FF$ if and only if, for every face $\sigma \in \Delta$ (including the empty face), $\tilde{\beta}_i(\lk_\Delta(\sigma))=0$ for all $i \ne d-1-|\sigma|$ (see \cite[Corollary II.4.2]{StGr}). Similarly, a simplicial complex is Buchsbaum over $\FF$ if and only if it is pure and all of its vertex links are Cohen--Macaulay over $\FF$ (see \cite[Theorem II.8.1]{StGr}).

A pure $(d-1)$-dimensional simplicial complex is  an \textbf{$\FF$-homology manifold without boundary} (or a \textbf{closed $\FF$-homology manifold}) if every proper link of $\Delta$, $\lk_\Delta(\sigma)$,  has the homology of a $(d-1-|\sigma|)$-dimensional sphere (over $\FF$).
Similarly, a pure $(d-1)$-dimensional simplicial complex $\Delta$ is an \textbf{$\FF$-homology manifold with boundary} if (i) every proper link of $\Delta$, $\lk_\Delta(\sigma)$,  has the homology of a $(d-1-|\sigma|)$-dimensional ball or a sphere (over $\FF$), and (ii) the boundary complex of $\Delta$, i.e.,
\[\partial(\Delta)=\{\sigma\in \Delta~:~ \tilde {H}_*(\lk_\Delta(\sigma);\FF)=0\}\cup\{\emptyset\},
\]
is an $\FF$-homology manifold without boundary.
An \textbf{$\FF$-homology $(d-1)$-sphere} is an $\FF$-homology manifold without boundary that has the same homology as the $(d-1)$-dimensional sphere, and an \textbf{$\FF$-homology $(d-1)$-ball} is an $\FF$-homology manifold with boundary whose homology is trivial and whose boundary complex is an $\FF$-homology $(d-2)$-sphere.
Thus, every proper link of an $\FF$-homology manifold with or without boundary is either an $\FF$-homology sphere or an $\FF$-homology ball.
In particular, if $\Delta$ is an $\FF$-homology manifold with or without boundary, then $\Delta$ is Buchsbaum over $\FF$.

We will often say that $(\Delta,\pi)$ is Cohen--Macaulay or Buchsbaum or an $\FF$-homology manifold if $\Delta$ has that property.

\subsection{Weak Lefschetz property}
Let $\Delta$ be a $(d-1)$-dimensional Cohen--Macaulay simplicial complex. 
We say that $\Delta$ has the \textbf{weak Lefschetz Property} (or \textbf{WLP}) over $\FF$ if there is an l.s.o.p.\ $\Theta$ for $\FF[\Delta]$ and a linear form $\omega$ such that the multiplication map
$$ \cdot \omega : (\FF[\Delta]/\Theta \FF[\Delta])_{\lfloor \frac d 2 \rfloor} \to (\FF[\Delta]/\Theta \FF[\Delta])_{\lfloor \frac d 2 \rfloor +1}$$
is surjective.
Similarly, we say that $\Delta$ has the \textbf{dual WLP} 
if there is an l.s.o.p.\ $\Theta$ for $\FF[\Delta]$ and a linear form $\omega$ such that the multiplication map
$$ \cdot \omega :(\FF[\Delta]/\Theta \FF[\Delta])_{\lfloor \frac {d+1} 2 \rfloor-1} \to (\FF[\Delta]/\Theta \FF[\Delta])_{\lfloor \frac {d+1} 2 \rfloor}$$
is injective.
While, in general, the above definitions differ from the usual definitions of the weak Lefschetz property (see \cite{LefschetzBook}), for homology spheres, our WLP coincides with the usual definition of the WLP; furthermore, in this case the WLP and the dual WLP are equivalent to each other (see \Cref{WLPdualWLP}).

The boundary complex of any simplicial $d$-polytope has the WLP over $\QQ$, and so does any triangulated $(d-1)$-ball that is a subcomplex of the boundary complex of a simplicial $d$-polytope (\cite{St-80} and \cite[Lemma 2.2]{St-93}).
It was repeatedly conjectured that all homology spheres and balls have the WLP.
While this conjecture is wide open at present, the following special case is well-known to be true.

\begin{lemma}
\label{2dimWLP}
All $\FF$-homology $2$-spheres and all $\FF$-homology $2$-balls have the WLP over $\FF$.
\end{lemma}

\noindent Indeed, in dimension 2, the class of $\FF$-homology spheres coincides with the class of triangulations of the (topological) $2$-sphere. 
The fact that triangulated $2$-spheres have the WLP over any field follows from \cite[Corollary 3.5]{Mu-10} and \cite{Whit}. For $\FF$-homology $2$-balls, the lemma is then derived exactly as in \cite[Lemma 2.2]{St-93}.


\section{Flag $h'$- and $h''$-vectors}\label{sect:flag h}
Over the last few decades, several refinements and modifications of $h$-vectors of simplicial complexes have been introduced and studied.  On one hand, already in 1979, Stanley \cite{St79} introduced $\aaa$-balanced simplicial complexes together with their flag $h$-vectors as a refinement of the classical $h$-vectors. Subsequently, these vectors have played an important role in the study of $f$-vectors of simplicial polytopes and simplicial complexes. On the other hand, in order to study face numbers of homology manifolds, one often considers their $h'$- and $h''$-vectors as certain modifications of the classical $h$-vector (cf.\ Section \ref{sect:Combinatorics}; also see \cite{KN-16,Swartz-14} for various applications of these combinatorial invariants).
Here we (i) combine these two approaches --- this results in the notions of \textbf{flag $h'$-} and \textbf{flag $h''$-vectors} of balanced simplicial complexes, and (ii) initiate the study of basic properties of these vectors. 
Most results in this section are natural extensions of known results on $h'$-, $h''$- and flag $h$-vectors, and so some details of proofs are omitted.

Let $(\Delta,\pi)$ be an $\aaa$-balanced simplicial complex.
We define the \textbf{flag $h'$-vector} $h'(\Delta,\pi)=(h'_\bb(\Delta,\pi) ~:~ \bb \leq \aaa)$ of $(\Delta,\pi)$ by
$$h'_\bb(\Delta,\pi)=h_\bb(\Delta,\pi) - {\aaa \choose \bb} \left( \sum_{j=1}^{|\bb|-1} (-1)^{|\bb|-j} \tilde{\beta}_{j-1}(\Delta)\right).$$ 
Schenzel \cite{Sc} proved that for a $(d-1)$-dimensional Buchsbaum complex $\Delta$ and an integer $0\leq j\leq d$, 
$\dim_\FF(\FF[\Delta]/\Theta\FF[\Delta])_j= h_j(\Delta)-{d\choose j}\sum_{i=1}^{j-1} (-1)^{j-i}\tilde{\beta}_{i-1}(\Delta)$ (see also \cite[Theorem II.8.2]{StGr}).
The following theorem establishes a flag analog of Schenzel's formula. 

\begin{theorem}
\label{2.3}
Let $(\Delta,\pi)$ be an $\aaa$-balanced simplicial complex and let $\Theta=\theta_1,\dots,\theta_{|\aaa|}$ be an $\NN^m$-graded l.s.o.p.\ for $\FF[\Delta]$. If $\Delta$ is Buchsbaum, then
$$\Hilb(\FF[\Delta]/\Theta \FF[\Delta];t_1,\dots,t_m)= \sum_{\bb \leq \aaa} h_\bb'(\Delta,\pi) \ttt^\bb.$$
\end{theorem}

\begin{proof}
We only sketch the proof since it is essentially the same as the proof of \cite[Theorem 4.3]{Sc}.
For a finitely generated graded $A$-module $M$, we denote by $H_\mideal^i(M)$ the $i$\textsuperscript{th} local cohomology module of $M$. Let $R=\FF[\Delta]$ and, for $S \subseteq \{1,2,\dots,|\aaa|\}$,
 let $\bdelta_S := \sum_{i \in S} \deg \theta_i$. 

Since $R$ is a Buchsbaum ring, we have the following exact sequences
\begin{align}
\label{2-1}
\begin{array}{lll}
0 &\longrightarrow H_\mideal^0(R/(\theta_1,\dots,\theta_{j-1})R)(-\deg \theta_j)
\longrightarrow (R/(\theta_1,\dots,\theta_{j-1})R)(-\deg \theta_j)\\
&\hspace{10pt} \stackrel{\times \theta_j} \longrightarrow R/(\theta_1,\dots,\theta_{j-1})R
\longrightarrow R/(\theta_1,\dots,\theta_{j})\longrightarrow 0
\end{array}
\end{align}
for $1\leq j \leq |\aaa|$.
Moreover, it follows from \cite[Lemma II.4.14$'$(i$'$)]{SV} that
\begin{align}
\label{2-2}
H_\mideal^0(R/(\theta_1,\dots,\theta_{j-1})R) 
\cong \bigoplus_{S \subseteq [j-1]} H_\mideal^{|S|} (R) (- \bdelta_S).
\end{align}
(While the statement given in \cite{SV} assumes the $\ZZ$-grading, the proof carries verbatim to the $\NN^m$-graded setting.)
Finally, since $R=\FF[\Delta]$ is Buchsbaum,
\begin{align}
\label{2-2-x}
H^i_\mideal(R)=(H^i_\mideal(R))_0\cong \tilde H_{i-1}(\Delta;\FF) \ \ \ \mbox{ for }i \leq \dim \Delta
\end{align}
(see \cite[Corollary II.4.13 and Lemma II.2.5(ii)]{SV}).
Combining Theorem \ref{2.1} with \eqref{2-1}, \eqref{2-2}, and \eqref{2-2-x},
we inductively obtain that
\begin{align*}
&\Hilb (R/(\theta_1,\dots,\theta_j) R;t_1,\dots,t_m)=\\
& \frac {\sum_{\bb \leq \aaa} h_\bb(\Delta,\pi) \ttt^\bb} {\prod_{i=j+1}^{|\aaa|} (1 - \ttt^{\deg \theta_i})}
- \sum_{\bb \leq \mbox{{\tiny $\bdelta_{[j]}$}} } {\mathbf \bdelta_{[j]} \choose \bb} 
\left(\sum_{k=1}^{|\bb|-1} (-1)^{|\bb|-k} \tilde \beta_{k-1}(\Delta) \right) \ttt^\bb
\end{align*}
for $1\leq j\leq |\aaa|$.
This proves the desired equation.
\end{proof}

We define the \textbf{flag $h''$-vector}
$h''(\Delta,\pi)=(h''_\bb(\Delta,\pi)~:~ \bb \leq \aaa)$ of an $\aaa$-balanced simplicial complex $(\Delta,\pi)$ by 
\begin{align*}
h_\bb''(\Delta,\pi)=
\left\{ \begin{array}{ll}
h_\bb(\Delta,\pi) - {\aaa \choose \bb} \big(\sum_{j=1}^{|\bb|}(-1)^{|\bb|-j}\tilde{\beta}_{j-1}(\Delta)\big),& \mbox{ if } \bb \ne \aaa,\\
\tilde \beta_{|\aaa|-1}(\Delta), & \mbox{ if } \bb=\aaa.
\end{array}
\right.
\end{align*}
We will see that the flag $h''$-vector is intimately related to the following ideal defined by Goto \cite{Go}. This ideal plays a crucial role in the paper.
For a finitely generated graded $A$-module $M$ of Krull dimension $d$ and an l.s.o.p.\ $\Theta=\theta_1,\dots,\theta_d$ for $M$,
let
$$\Sigma(\Theta;M):= \Theta M +\sum_{i=1}^d  \left((\theta_1,\dots,\hat \theta_i,\dots,\theta_d)M :_M \theta_i\right).$$
Here, $\hat \theta_i$ indicates that $\theta_i$ is omitted from $\Theta$. Note that if $M$ and $\Theta$ are $\NN^m$-graded, then so are $M/\Theta M$ and $M/\Sigma(\Theta; M)$.

The following property was essentially proved by Goto in the setting of local rings. The proof for the $\NN$-graded case can be found in \cite[Theorem 2.3]{MNY} and it extends easily  to the $\NN^m$-graded setting.

\begin{theorem}[Essentially Goto \cite{Go}]
\label{2.4}
Let $M$ be a finitely generated $\NN^m$-graded Buchsbaum $A$-module of Krull dimension $d$ and let $\Theta=\theta_1,\dots,\theta_d$ be an $\NN^m$-graded l.s.o.p.\ for $M$. Then there is an isomorphism of $A$-modules
$$\Sigma(\Theta; M) / \Theta M \cong \bigoplus_{S \subsetneq [d]} H_\mideal^{|S|} (M) (- {\textstyle \sum_{k \in S} \deg \theta_k}).$$
In particular,
$\Sigma(\Theta; M)/\Theta M$ is contained in the socle of $M/\Theta M$. In other words, 
$\mideal \cdot (\Sigma(\Theta; M)/\Theta M) =0$.
\end{theorem}

Note that the latter statement of the previous theorem follows from the fact that $\mideal\cdot H_\mideal^i(M)=0$
holds for any Buchsbaum $A$-module $M$ and $i < \dim M$
(see \cite[Proposition I.2.1 (iii)]{SV}).
Since $H_\mideal^i(\FF[\Delta])=(H_\mideal^i(\FF[\Delta]))_0 \cong \tilde H_{i-1}(\Delta;\FF)$ for $i \leq \dim \Delta$ when $\Delta$ is Buchsbaum,
Theorem \ref{2.4} implies the following result.

\begin{corollary}
\label{2.5}
Let $(\Delta,\pi)$ be an $\aaa$-balanced simplicial complex and let $\Theta$ be an $\NN^m$-graded l.s.o.p.\ for $\FF[\Delta]$.
If $\Delta$ is Buchsbaum over $\FF$, then
$$\dim_\FF\big(\Sigma(\Theta;\FF[\Delta])/\Theta \FF[\Delta]\big)_\bb = { \aaa \choose \bb} \tilde \beta_{|\bb|-1}(\Delta) \ \ \mbox{ for any $\bb \lneq \aaa$}.$$
\end{corollary}

Observe that for $\bb\neq\aaa$, $h''_\bb(\Delta,\pi)=h'_\bb(\Delta,\pi) - {\aaa \choose \bb} \tilde \beta_{|\bb|-1}(\Delta)$ and that $h''_\aaa(\Delta)=h'_\aaa(\Delta)$.
Since $\FF[\Delta]/\Theta\FF[\Delta]$ and $\FF[\Delta]/\Sigma(\Theta;\FF[\Delta]) \oplus \Sigma(\Theta;\FF[\Delta])/\Theta\FF[\Delta]$ are isomorphic as graded vector spaces, \Cref{2.3} and \Cref{2.5} lead to the following algebraic interpretation of flag $h''$-vectors.

\begin{theorem}
\label{2.6}
Let $(\Delta,\pi)$ be an $\aaa$-balanced simplicial complex and let $\Theta$ be an $\NN^m$-graded l.s.o.p.\ for $\FF[\Delta]$. If $\Delta$ is Buchsbaum over $\FF$, then
$$\Hilb(\FF[\Delta]/\Sigma(\Theta;\FF[\Delta]);t_1,\dots,t_m)= \sum_{ \bb \leq \aaa} h_\bb''(\Delta,\pi) \ttt^\bb.$$
\end{theorem}

\begin{remark}
The flag $h'$- and $h''$-numbers refine the usual  $h'$- and $h''$-numbers: 
$$h_i'(\Delta)=\sum_{\bb \leq \aaa, |\bb|=i} h_\bb'(\Delta,\pi) \quad \mbox{and} \quad
h_i''(\Delta)=\sum_{\bb \leq \aaa, |\bb|=i} h_\bb''(\Delta,\pi).$$
\end{remark}

\begin{remark}
\label{2.8}
All results on flag $h'$- and $h''$-vectors in this section are natural extensions of known results on the usual $h'$- and $h''$-vectors.
In the $\NN$-graded setting,
the formula for $h'$-vectors in Theorem \ref{2.3} was proved in \cite{Sc}, and the formula for $h''$-vectors in Theorem \ref{2.6} was given in \cite{MNY}.
When $\Delta$ is a Cohen--Macaulay simplicial complex, the
flag $h'$- and $h''$-vectors coincide with the usual flag $h$-vectors. In this case the formula for the Hilbert series in Theorem \ref{2.3} is due to Stanley \cite{St79}.
\end{remark}

The results in this section continue to hold in the generality of Stanley--Reisner modules of relative simplicial complexes.
Here we quickly review some relevant notions.

For a simplicial complex $\Delta$ with the vertex set $V$ and a subcomplex $\Gamma$ of $\Delta$,
the $A$-module
$$\FF[\Delta,\Gamma]= I_{\Gamma}/I_\Delta
$$
is called the \textbf{Stanley--Reisner module} of the pair $(\Delta,\Gamma)$,
where we consider $I_\Gamma=(x_\sigma~:~\sigma \subseteq V, \sigma \not \in \Gamma)$ as an ideal of $A$.
The faces of $(\Delta,\Gamma)$ are the elements of $\Delta\setminus\Gamma$. With this convention in hand, 
we define the $f$-, $h$-, $h'$- and $h''$-vector of the pair $(\Delta,\Gamma)$ as well as the flag $f$-, $h$-, $h'$- and $h''$-vectors in the same way as for a single simplicial complex.
In particular, for an $\aaa$-balanced simplicial complex $(\Delta,\pi)$ and its subcomplex $\Gamma$,
$f_\bb(\Delta,\Gamma,\pi)$ is the number of faces $\sigma \in \Delta \setminus \Gamma$ with $\pi(\sigma)=\bb$ and
\begin{align*}
h_\bb''(\Delta,\Gamma,\pi)=
\left\{ \begin{array}{ll}
h_\bb(\Delta,\Gamma,\pi) - {\aaa \choose \bb} \big(\sum_{j=1}^{|\bb|}(-1)^{|\bb|-j}\tilde{\beta}_{j-1}(\Delta,\Gamma)\big),& \mbox{ if } \bb \ne \aaa,\\
\tilde \beta_{|\aaa|-1}(\Delta,\Gamma), & \mbox{ if } \bb=\aaa,
\end{array}
\right.
\end{align*}
where $\tilde{\beta}_i(\Delta,\Gamma):=\dim_\FF\tilde {H}_i(\Delta,\Gamma;\FF)$.

Theorem \ref{2.3}, Corollary \ref{2.5}, and Theorem \ref{2.6}, in fact, hold for all Buchsbaum Stanley--Reisner modules. (We omit the proofs since they are identical to the proofs above, except that the notation becomes somewhat more cumbersome).
Specifically, if $\Delta$ is an $\FF$-homology manifold with boundary, then $\FF[\Delta,\partial \Delta]$ is Buchsbaum, and hence the statements of Theorem \ref{2.3}, Corollary \ref{2.5}, and Theorem \ref{2.6} continue to hold in this setting but with $\Delta$ replaced throughout by $(\Delta,\partial \Delta)$.

\section{Proofs in the orientable case}
The purpose of this section is to prove \Cref{BGLBT} for \emph{orientable} homology manifolds with and without boundary, see \Cref {thm:GLBT last step}. If $\Delta$ is a homology manifold without boundary,
we will often identify $\Delta$ with the pair $(\Delta,\partial \Delta)$ where we let $\partial \Delta=\emptyset$.
We say that an $\FF$-homology manifold $\Delta$ with or without boundary is \textbf{orientable} if the top Betti number of $(\Delta,\partial \Delta)$ computed over $\FF$ is equal to the number of connected components of $\Delta$.

Let $(\Delta,\pi)$ be a balanced $(d-1)$-dimensional $\FF$-homology manifold. If $\Delta$ has no boundary, then
the link of each codimension-1 face of $\Delta$ consists of two vertices. Thus, for each color $i$, in each connected component there exist at least two vertices of color $i$, so that $f_0(\Delta)\geq 2d(1+\tilde{\beta}_0(\Delta))$. Hence $h''_1(\Delta)=f_0(\Delta)-d-d\tilde{\beta}_0(\Delta)\geq d(1+\tilde{\beta}_0(\Delta))$. Since $h''_0(\Delta)= 1$, the inequality $\frac{h''_1(\Delta)}{d}\geq \frac{h''_0(\Delta)}{1}+\tilde{\beta}_0(\Delta)$ follows. Similarly, if $\Delta$ has non-empty boundary, then $h''_0(\Delta,\partial \Delta)=0$, $\tilde \beta_0(\Delta,\partial \Delta)$ equals the number of connected components of $\Delta$ that have no boundary, and each component with non-empty boundary has at least $d$ vertices. The same computation as above then shows that $h''_1(\Delta,\partial\Delta)\geq d\tilde \beta_0(\Delta,\partial \Delta)$. We conclude that \Cref{BGLBT} holds for $\ell=1$, and from now on,  assume that $1<\ell\leq \lfloor d/2\rfloor$. 

Moreover, it suffices to prove \Cref{BGLBT} for connected $\FF$-homology manifolds. Indeed, a straightforward computation shows that if $\Delta$ is disconnected with connected components $\Delta^1,\ldots,\Delta^s$, then $h''_j(\Delta, \partial\Delta)=\sum_{k=1}^s h''_j(\Delta^k, \partial\Delta^k)$ for all $j\geq 1$; in addition, $\tilde{\beta}_{j}(\Delta, \partial\Delta)=\sum_{k=1}^s\tilde{\beta}_{j}(\Delta^k, \partial\Delta^k)$ for $j\geq 1$. (Furthermore, if  $\Delta$ is orientable, then so is each connected component of $\Delta$.) Therefore, if each connected component of $\Delta$ satisfies the inequality in \Cref{BGLBT}, then so does $\Delta$.

For a graded $A$-module $N$,
let $N^\vee$ denote the Matlis dual of $N$.
One crucial ingredient in the proof of \Cref{BGLBT} is the following result established in \cite[Corollary 1.4]{MNY}.
Although, in \cite{MNY}, only the $\NN$-graded case is treated, the same proof works in the $\NN^m$-graded setting.


\begin{theorem}[Murai--Novik--Yoshida] \label{DeterminingShift}
Let $\aaa=(a_1,\ldots,a_m)\in \NN^m$, let $(\Delta,\pi)$ be an $\aaa$-balanced, connected, orientable $\FF$-homology manifold with or without boundary, and let $\Theta$ be an $\NN^m$-graded l.s.o.p.\ for $\FF[\Delta]$. Then 
\[
\big(\FF[\Delta,\partial \Delta]/\Sigma(\Theta; \FF[\Delta,\partial \Delta])\big)\cong\big(\FF[\Delta]/\Sigma(\Theta; \FF[\Delta])\big)^\vee (-\aaa).
\]
In particular, $h''_{\bb}(\Delta,\partial \Delta,\pi) =h''_{\aaa-\bb}(\Delta,\pi)$ for all $\bb\in \NN^m$ with $\bb\leq \aaa$.
\end{theorem}

The proof of \Cref{BGLBT} will essentially follow from the next proposition.

\begin{proposition} \label{numerics}
Let $1<\ell\leq \lfloor d/2 \rfloor$, $\aaa=(2\ell-1,1,\ldots,1)\in\NN^{d-2\ell+2}$, and $\bb=\aaa-(2\ell-1)\ee_1$.
Let $(\Delta,\pi)$ be an $\aaa$-balanced, connected, orientable $\FF$-homology manifold with or without boundary and suppose that for every face $\sigma\in\Delta$ with $\pi(\sigma)=\bb$, the link of $\sigma$ in $\Delta$ has the WLP. Then 
\[
  h''_{\ell\ee_1}(\Delta,\partial \Delta,\pi)- h''_{(\ell-1)\ee_1}(\Delta,\partial \Delta,\pi)\geq {2\ell-1\choose \ell}\tilde{\beta}_{\ell-1}(\Delta,\partial \Delta).
  \]
\end{proposition}

\begin{proof}
Let $\Theta=\theta_1,\ldots,\theta_d$ be an $\NN^{d-2\ell+2}$-graded l.s.o.p.~for $\FF[\Delta]$ and let $\Theta':=\{\theta_i~:~\deg\theta_i=\ee_1\}$.  Consider the following modules: 
 \[ L=\bigoplus_{\sigma\in\Delta, \pi(\sigma)=\bb}\FF[\lk_\Delta(\sigma)]/\Theta'\FF[\lk_\Delta(\sigma)] \;\mbox{ and }\;  M=(\FF[\Delta]/\Theta\FF[\Delta])_{\ge \bb}.
 \]

In analogy to \cite[Lemma 2.3 (i)]{JM}, there exists a surjection $\psi:L\to M(\bb)$.
Thus for any linear form  $\omega\in \FF[\Delta]$ with $\deg\omega=\ee_1$, there is the following commutative diagram: 
 \begin{align*}
 \begin{array}{rcl}
 \displaystyle
 L_{\ell\ee_1} & \stackrel {\psi} \longrightarrow & M(\bb)_{\ell\ee_1}=\left(\FF[\Delta]/\Theta\FF[\Delta]\right)_{\ell\ee_1+\bb}\medskip\\
 \cdot \omega \uparrow \hspace{10pt}& &\hspace{ 10pt} \uparrow \cdot \omega\medskip\\
 \displaystyle
L_{(\ell-1)\ee_1} & \stackrel {\psi} \longrightarrow & M(\bb)_{(\ell-1)\ee_1}=\left(\FF[\Delta]/\Theta\FF[\Delta]\right)_{(\ell-1)\ee_1+\bb}.
 \end{array}
 \end{align*}
Note that all links, $\lk_\Delta(\sigma)$, in the above diagram are monochromatic (indeed they are $(2\ell-1)$-balanced).
Since $\lk_\Delta(\sigma)$ has the WLP for all $\sigma\in\Delta$ with $\pi(\sigma)=\bb$ by assumption, it follows that the left multiplication map  $\cdot \omega$ is surjective for a generic choice of $\Theta$ and $\omega$. 
This fact and the surjectivity of the horizontal maps $\psi$ implies that the multiplication map $\cdot \omega$ on the right is also surjective. Furthermore, since $\Sigma(\Theta; \FF[\Delta])/\Theta\FF[\Delta]$ is contained in the socle of $\FF[\Delta]/\Theta\FF[\Delta]$ (see \Cref{2.4}), we infer that the map 
\[
\cdot\omega : \left(\FF[\Delta]/\Sigma(\Theta;\FF[\Delta])\right)_{(\ell-1)\ee_1+\bb}\rightarrow \left(\FF[\Delta]/\Theta\FF[\Delta]\right)_{\ell\ee_1+\bb}
\]
is well-defined and surjective. Consequently,
\begin{equation} \label{star} \dim_\FF\left( \FF[\Delta]/\Sigma(\Theta;\FF[\Delta])\right)_{(\ell-1)\ee_1+\bb}\geq 
\dim_\FF \left(\FF[\Delta]/\Theta\FF[\Delta]\right)_{\ell\ee_1+\bb}.
\end{equation}

To finish the proof, we compute both sides of \eqref{star}. Theorems~\ref{2.6} and \ref{DeterminingShift} imply
\begin{equation} \label{2star}
\dim_\FF\left( \FF[\Delta]/\Sigma(\Theta;\FF[\Delta])\right)_{(\ell-1)\ee_1+\bb}=h''_{(\ell-1)\ee_1+\bb}(\Delta,\pi)=h''_{\ell\ee_1}(\Delta,\partial \Delta,\pi),
\end{equation}
while by \Cref{2.3}, 
\begin{align} \nonumber
\dim_\FF \left(\FF[\Delta]/\Theta\FF[\Delta]\right)_{\ell\ee_1+\bb}=&h'_{\ell\ee_1+\bb}(\Delta,\pi)\\
\nonumber
=&h''_{\ell\ee_1+\bb}(\Delta,\pi)+{2\ell-1\choose \ell-1}\tilde{\beta}_{d-\ell}(\Delta)\\
\label{3star}
=&h''_{(\ell-1)\ee_1}(\Delta,\partial \Delta,\pi)+{2\ell-1\choose \ell-1}\tilde{\beta}_{\ell-1}(\Delta,\partial \Delta).
\end{align}
Here the last step follows from \Cref{DeterminingShift} and Poincar\'e--Lefschetz duality asserting that $\tilde \beta_{\ell-1}(\Delta,\partial \Delta) = \tilde{\beta}_{d-\ell}(\Delta)$. Substituting \eqref{2star} and \eqref{3star} in \eqref{star} yields the result.
\end{proof}

For a balanced simplicial complex $(\Delta,\pi)$
of dimension $d-1$,  a subcomplex $\Gamma$ of $\Delta$, and $S\subseteq [d]$,
define
$$h_S(\Delta,\Gamma,\pi)= h_{\ee_S}(\Delta,\Gamma,\pi) \mbox{ and }
h''_S(\Delta,\Gamma,\pi)= h''_{\ee_S}(\Delta,\Gamma,\pi).$$
We  also define the \textbf{normalized $h''_i$-number} of the pair $(\Delta,\Gamma)$, $\hav''_i(\Delta,\Gamma)$, by
$$
\hav''_i(\Delta,\Gamma)= \frac {h''_i(\Delta,\Gamma)} {{d \choose i}} = \frac {\sum_{ S \subseteq [d],|S|=i} h''_S (\Delta,\Gamma,\pi)} {{d \choose i}}.
$$

The following lemma is an easy consequence of \cite[Lemma 3.6]{JM}; we omit the proof.
\begin{lemma}
\label{average}
Let $(\Delta,\pi)$ be a $(d-1)$-dimensional balanced simplicial complex and $\Gamma$ a subcomplex of $\Delta$. Then for any $1 \leq \ell \leq d/2$,
\begin{align*}
&h''_\ell(\Delta,\Gamma)-h''_{\ell-1}(\Delta,\Gamma)=\\
&\frac{1}{{2\ell-1\choose \ell}{d\choose 2\ell-1}}\left(\sum_{S\subseteq [d],|S|=2\ell-1}\left(\sum_{T\subseteq S,|T|=\ell}h''_{T}(\Delta,\Gamma,\pi)-\sum_{T\subseteq S, |T|=\ell -1}h''_T(\Delta,\Gamma,\pi)\right)\right).
\end{align*}
\end{lemma}

We are now in a position to prove the main result of this section.

\begin{theorem}\label{thm:GLBT last step}
Let $1 < \ell \leq \lfloor d/2 \rfloor$.
Let $(\Delta,\pi)$ be a balanced orientable $\FF$-homology manifold with or without boundary of dimension $d-1$.
Suppose that for all faces $\sigma\in \Delta$ of codimension-$(2\ell-1)$, the link of $\sigma$ has the WLP. Then
\begin{itemize}
\item[(i)] for any $S\subseteq [d]$ with $|S|=2\ell -1$, 
\[\sum_{T\subseteq S,|T|=\ell}h''_T(\Delta,\partial \Delta,\pi)-\sum_{T\subseteq S,|T|=\ell-1}h''_T(\Delta,\partial \Delta,\pi)\geq {2\ell-1\choose \ell}\tilde{\beta}_{\ell-1}(\Delta,\partial \Delta).\]
\item[(ii)] Consequently,  $\displaystyle\hav''_\ell(\Delta,\partial \Delta)-\hav''_{\ell-1}(\Delta,\partial \Delta)\geq \tilde{\beta}_{\ell-1}(\Delta,\partial \Delta).$
\end{itemize}
\end{theorem}

\begin{proof}
First note that part (ii) of the statement is an immediate consequence of part (i) and \Cref{average}. To prove part (i),
we may assume that $\Delta$ is connected since for $T \ne \emptyset$, the number $h''_T(\Delta,\partial \Delta,\pi)$ is the sum of the corresponding statistics of connected components of $\Delta$. Furthermore, by relabeling the vertices, we may assume that $S=\{1,d-2\ell+3,d-2\ell+4,\ldots,d\}$. Define $\tilde{\pi}:V(\Delta)\rightarrow \{\ee_1,\ldots,\ee_{d-2\ell+1}\}$ by $\tilde{\pi}(v)=\pi(v)$ if $\pi(v)\notin \{\ee_i~:~i\in S\}$ and $\tilde{\pi}(v)=\ee_1$ if $\pi(v)\in \{\ee_i~:~i\in S\}$. Then $(\Delta,\tilde{\pi})$ is $(2\ell-1,1,\ldots,1)$-balanced. As
$h''_{\ell\ee_1}(\Delta,\partial \Delta,\tilde{\pi})=\sum_{T\subseteq S,|T|=\ell}h''_T(\Delta,\partial \Delta,\pi)$
and $h''_{(\ell-1)\ee_1}(\Delta,\partial \Delta, \tilde{\pi})=\sum_{T\subseteq S,|T|=\ell-1}h''_T(\Delta,\partial \Delta,\pi)$,
the claim follows from \Cref{numerics}.
\end{proof}

Since all proper links of a homology manifold with or without boundary are homology spheres or homology balls, we infer from 
\Cref{2dimWLP} the following result.

\begin{corollary}
Let $\Delta$ be a  balanced orientable $\FF$-homology manifold with non-empty boundary. If $\Delta$ has dimension $\geq 3$, then
$\hav''_2(\Delta,\partial \Delta)-\hav''_1(\Delta,\partial \Delta)\geq \tilde{\beta}_1(\Delta,\partial \Delta).$
\end{corollary}


\section{Canonical modules}
Our proof of \Cref{BGLBT}(i) for non-orientable homology manifolds relies on canonical modules. This requires a few auxiliary results on canonical modules, some of which are discussed in this section. 

Recall that if $M$ is a finitely generated graded $A$-module of Krull dimension $d$, then the \textbf{canonical module} of $M$ is the module
$$\cano M:=H_\mideal^d(M)^\vee.$$ In particular, for an $\aaa$-balanced simplicial complex $(\Delta,\pi)$ with $\aaa\in \NN^m$, the canonical module of $\FF[\Delta]$ is $\NN^m$-graded.

We start by reviewing some dualities that are exhibited by canonical modules of Buchsbaum rings. 
The following is an algebraic generalization of Theorem \ref{DeterminingShift} above, proved in \cite[Theorem 1.3]{MNY}.

\begin{theorem}[Murai--Novik--Yoshida]
\label{dual}
Let $(\Delta,\pi)$ be an $\aaa$-balanced Buchsbaum simplicial complex with $|\aaa|\geq 2$. Let $\Theta=\theta_1,\ldots,\theta_{|\aaa|} \in \FF[\Delta]$ be an $\NN^m$-graded l.s.o.p.\ for $\FF[\Delta]$. If $\Delta$ is connected, then
\begin{equation*}
\cano{\FF[\Delta]}/\Sigma(\Theta;\cano{\FF[\Delta]})\cong \big(\FF[\Delta]/\Sigma(\Theta;\FF[\Delta])\big)^{\vee}(-\aaa).
\end{equation*}
\end{theorem}

\noindent We note that, in the above statement, $\Theta$ is automatically also an l.s.o.p.\ for $\cano {\FF[\Delta]}$.

\begin{remark}
\label{dualCM}
If $M$ is a Cohen--Macaulay $A$-module and $\Theta$ is an l.s.o.p.\ for $M$, then $\Sigma(\Theta;M)=\Theta M$ by definition of the Cohen--Macaulay property. Thus, if $\Delta$ is Cohen--Macaulay, then \Cref{dual} gives an isomorphism 
$$\cano{\FF[\Delta]}/\Theta \cano{\FF[\Delta]} \cong \big(\FF[\Delta]/\Theta \FF[\Delta]\big)^\vee (-\aaa),$$ a fact that is well-known in commutative algebra.
\end{remark}

The above duality (for the monochromatic case) implies the following equivalent formulation of the dual WLP; we will use it in the next section.

\begin{lemma} \label{dualwlp}
Let $\Delta$ be a Cohen--Macaulay simplicial complex of dimension $d-1$.
Then $\Delta$ has the dual WLP if and only if there is an l.s.o.p.\ $\Theta$ for $\FF[\Delta]$ and a linear form $\omega$ such that the multiplication map
\[ \cdot \omega : (\cano{\FF[\Delta]}/\Theta \cano{\FF[\Delta]})_{\lfloor d/2 \rfloor}
\to (\cano{\FF[\Delta]}/\Theta \cano{\FF[\Delta]})_{\lfloor d/2 \rfloor+1}
\]
is surjective.
\end{lemma}

\begin{remark}
\label{WLPdualWLP}
If $\Delta$ is an $\FF$-homology sphere, then $\cano{\FF[\Delta]}\cong \FF[\Delta]$.
Hence, in the case of $\FF$-homology spheres, having the WLP is equivalent to having the dual WLP.
\end{remark}

We will also use the following duality result due to Schenzel \cite[Theorem II.4.9]{SV}. 

\begin{theorem}[{Schenzel}]
\label{thm:cano2}
Let $R$ be a finitely generated graded $\FF$-algebra of Krull dimension $d>0$.
If $R$ is Buchsbaum, then $\Omega(R)$ is also Buchsbaum and
$$H_{\mathfrak{m}}^i(\cano R)\cong (H_{\mathfrak{m}}^{d-i+1}(R))^\vee \quad \mbox{for all  $2\leq i\leq d-1$}.$$
\end{theorem}

One of the key properties used in the proof of \Cref{thm:GLBT last step} (see the proof of Proposition \ref{numerics}) was the existence of a surjection from $L$ to $M(\bb)$.
The goal of the rest of this section is to establish an analogous surjection for canonical modules. 

In the rest of this section,
we assume that $(\Delta,\pi)$ is an $\aaa$-balanced simplicial complex on the vertex set $[n]$, and we consider $\FF[\Delta]$, $\FF[\st_\Delta(\sigma)]$ and $\FF[\lk_\Delta(\sigma)]$ as modules over the polynomial ring $A=\FF[x_1,\dots,x_n]$.
We utilize two different fine gradings of $A$: the $\NN^m$-grading induced by the coloring $\pi$ and the $\NN^n$-grading  defined by $\deg x_i= \eenu_i$, where $\eenu_1,\dots,\eenu_n$ is the standard basis for $\ZZ^n$. 
To avoid confusion, we use bold letters for elements in $\NN^m$ whereas we use letters in Fraktur for elements in $\NN^n$.
For $\sigma \subseteq [n]$, we set $\eenu_\sigma=\sum_{i \in \sigma} \eenu_i$.

\begin{lemma}
\label{3.0}
Let $\aaa,\bb \in \NN^m$ with $\bb \leq \aaa$.
Let $(\Delta,\pi)$ be an $\aaa$-balanced simplicial complex and let $\Theta=\theta_1,\dots,\theta_{|\aaa|}$ be an $\NN^m$-graded l.s.o.p.\ for $\FF[\Delta]$.
Then, there is a surjection
$$\bigoplus_{\sigma \in \Delta,  \pi(\sigma)=\bb}
\cano{\FF[\st_\Delta(\sigma)]} \to \cano{\FF[\Delta]}_{\geq \bb}.$$
\end{lemma}

\begin{proof}
Let $|\aaa|=d$ and let $\sigma \in \Delta$ be any face. 
The long exact sequence of local cohomology modules induced by the natural surjection $\FF[\Delta] \to \FF[\st_\Delta(\sigma)]$  provides us with a surjection
$$H_\mideal^d(\FF[\Delta]) \to H^d_\mideal(\FF[\st_\Delta(\sigma)]).$$
By taking the Matlis dual of both sides, we obtain an injection
\begin{align}
\label{3-1-1}
\cano{\FF[\st_\Delta(\sigma)]} \to \cano{\FF[\Delta]} \quad \mbox{for all $\sigma \in \Delta$.}
\end{align}

When $\mathfrak u=(u_1,\dots,u_n) \in \NN^n$, we let $\supp(\mathfrak u)=\{i ~:~ u_i \ne 0\}$ be the support of $\mathfrak{u}$. 
If $\supp(\mathfrak u) \supseteq \sigma$, then Hochster's formula for the Hilbert series of local cohomology modules \cite[Theorem II.4.1]{StGr} shows that
$$\cano{\FF[\Delta]}_{\mathfrak u}
\cong \tilde H_{d-1-|\supp(\mathfrak u)|} \big(\lk_\Delta(\supp(\mathfrak u));\FF\big) = \tilde H_{d-1-|\supp(\mathfrak u)|} \big(\lk_{\st_\Delta(\sigma)}(\supp(\mathfrak u));\FF\big) \cong \cano{\FF[\st_\Delta(\sigma)]}_{\mathfrak u}.$$
If $\supp(\mathfrak u)\not\supseteq \sigma$, then $\lk_{\st_\Delta(\sigma)}(\supp(\mathfrak u))$ is a cone over $\sigma\smallsetminus \supp(\mathfrak u)$, and hence has trivial homology. It hence follows from  Hochster's formula that
$\Omega(\FF[\st_\Delta(\sigma)])$ is equal to $\Omega(\FF[\st_\Delta(\sigma)])_{\geq \eenu_\sigma}$. 
Thus, for any $\sigma \in \Delta$, the injection in \eqref{3-1-1} induces an isomorphism
\begin{align}
\label{3-5-2}
\cano{\FF[\st_\Delta(\sigma)]}=\cano{\FF[\st_\Delta(\sigma)]}_{\geq \eenu_\sigma} \to \cano{\FF[\Delta]}_{\geq \eenu_\sigma}.
\end{align}
Let  $$\mathcal L_\bb=\{ \mathfrak u \in \NN^n~:~ \mathfrak u \geq \eenu_\sigma \mbox{ for some } \sigma \in \Delta \mbox{ with } \pi (\sigma)=\bb\}.$$
Taking the following sum of the maps in \eqref{3-5-2} yields the desired surjection
\begin{align*}
\bigoplus_{\sigma \in \Delta, \pi(\sigma)=\bb}
 \cano{\FF[\st_\Delta(\sigma)]} \to \bigoplus_{\mathfrak u  \in \mathcal L_\bb}
\cano{\FF[\Delta]}_{\mathfrak u} = \cano{\FF[\Delta]}_{\geq \bb}.
\end{align*}
\end{proof}

The following modification of \Cref{3.0} provides an appropriate analog of a surjection from $L$ to $M(\bb)$ on the level of canonical modules.


\begin{lemma}
\label{3.1}
Let $0 \leq \ell \leq m$, $\aaa=(a_1,\dots,a_m) \in \NN^m$, and $\bb=(a_1,\dots,a_\ell,0,\dots,0)\in\NN^m$.
Let $(\Delta,\pi)$ be an $\aaa$-balanced simplicial complex, $\Theta=\theta_1,\dots,\theta_{|\aaa|}$ an $\NN^m$-graded l.s.o.p.\ for $\FF[\Delta]$, and $\Theta'=(\theta_i~:~ \deg \theta_i \not \in \{\ee_1,\dots,\ee_\ell\})$.
Then there is a surjection
\begin{equation*}
 \bigoplus_{\sigma \in \Delta, \pi(\sigma)=\bb} \big( \cano{\FF[\lk_\Delta(\sigma)]}/\Theta' \cano{\FF[\lk_\Delta(\sigma)]} \big) (-\bb) \to \big( \cano{\FF[\Delta]}/\Theta \cano{\FF[\Delta]}\big)_{\geq \bb}.
\end{equation*}
\end{lemma}

\begin{proof}
By \Cref{3.0} there exists a surjection
$$\bigoplus_{\sigma \in \Delta, \pi (\sigma)=\bb} \cano{\FF[\st_\Delta(\sigma)]}/
\Theta\cano{\FF[\st_\Delta(\sigma)]}
\to (\cano{\FF[\Delta]}/\Theta \cano{\FF[\Delta]})_{\geq \bb}.$$
Thus, in order to prove the claim, it is enough to show that for any $\sigma \in \Delta$ with $\pi(\sigma)=\bb$, there is an isomorphism
\begin{equation}\label{iso}
\cano{\FF[\st_\Delta(\sigma)]}/\Theta \cano{\FF[\st_\Delta(\sigma)]} \cong
\big(\cano{\FF[\lk_\Delta(\sigma)]}
/\Theta'\cano{\FF[\lk_\Delta(\sigma)]}\big)(-\bb).
\end{equation}

Fix a face $\sigma \in \Delta$ with $\pi(\sigma)=\bb$. Since the variables $x_v$ with $v  \in \sigma$ form a regular sequence of $\FF[\st_\Delta(\sigma)]$
and since $\FF[\st_\Delta(\sigma)]/(x_v: v \in \sigma) \FF[\st_\Delta(\sigma)] \cong \FF[\lk_\Delta(\sigma)]$,
it follows from \cite[Theorem 3.3.5]{BRH} that
\begin{align}
\label{3-1-3}
\textstyle
\cano{\FF[\st_\Delta(\sigma)]}/(x_v: v \in \sigma)\cano{\FF[\st_\Delta(\sigma)]} \cong \cano{\FF[\lk_\Delta(\sigma)]}(-\bb).
\end{align} 
Also, since $\Theta$ contains exactly $a_i$ linear forms of color $\ee_i \in \NN^m$ for all $i$ and since $\st_\Delta(\sigma)$ contains exactly $a_i$ vertices of color $\ee_i \in \NN^m$ for all $i \leq\ell$,
we obtain that
$$\Theta \FF[\st_\Delta(\sigma)]=\big((\Theta')+(x_v: v \in \sigma)\big) \FF[\st_\Delta(\sigma)].$$
Then, since $\cano{\FF[\st_\Delta(\sigma)]}$ is an $\FF[\st_\Delta(\sigma)]$-module, we conclude that
\begin{align}
\label{3-1-4}
\Theta \cano{\FF[\st_\Delta(\sigma)]} = \Theta' \cano{\FF[\st_\Delta(\sigma)]}+ (x_v: v \in \sigma) \cano{\FF[\st_\Delta(\sigma)]}.
\end{align}
Combining \eqref{3-1-3} and \eqref{3-1-4} yields the desired isomorphism \eqref{iso}.
\end{proof}

\section{Proofs for non-orientable manifolds}
In this section we prove Theorem \ref{BGLBT}(i) for non-orientable homology manifolds, see \Cref{3.4}.
As in Section 4, we may assume that $\ell>1$ and that $\Delta$ is connected.
We start with the following result (cf. \Cref{numerics}).

\begin{proposition} \label{3.3}
Let $1<\ell\leq \lfloor d/2 \rfloor$, $\aaa=(2\ell-1,1,\ldots,1)\in\NN^{d-2\ell+2}$, and $\bb=\aaa-(2\ell-1)\ee_1$.
Let $(\Delta,\pi)$ be an $\aaa$-balanced, connected, Buchsbaum simplicial complex and suppose that for every face $\sigma\in\Delta$ with $\pi(\sigma)=\bb$, the link of $\sigma$ in $\Delta$ has the dual WLP. Then 
\[
  h''_{\ell\ee_1}(\Delta,\pi)- h''_{(\ell-1)\ee_1}(\Delta,\pi)\geq {2\ell-1\choose \ell}\tilde{\beta}_{\ell-1}(\Delta).
  \]
\end{proposition}

\begin{proof}
The proof is similar to that of \Cref{numerics}.
For any $\NN^m$-graded l.s.o.p.\ $\Theta=\theta_1,\dots,\theta_{|\aaa|}$ for $\FF[\Delta]$ and for any linear form $\omega$ with $\deg \omega=\ee_1$,
there is the following commutative diagram:
{\small
\begin{align*}
\begin{array}{ccc}
\displaystyle
\bigoplus_{\sigma \in \Delta, \pi(\sigma)=\bb} \hspace{-15pt} \big(\cano{\FF[\lk_\Delta(\sigma)]}/\Theta' \cano{\FF[\lk_\Delta(\sigma)]} \big)_{\ell\ee_1} & \stackrel {\varphi} \longrightarrow & \big( \cano{\FF[\Delta]}/\Theta \cano{\FF[\Delta]}\big)_{\ell \ee_1+ \bb}\\
\cdot \omega \uparrow & & \uparrow \cdot \omega\medskip\\
\displaystyle
\bigoplus_{\sigma \in \Delta, \pi(\sigma)=\bb} \hspace{-15pt} \big(\cano{\FF[\lk_\Delta(\sigma)]}/\Theta' \cano{\FF[\lk_\Delta(\sigma)]} \big)_{(\ell-1)\ee_1} & \stackrel {\varphi} \longrightarrow & \big( \cano{\FF[\Delta]}/\Theta \cano{\FF[\Delta]} \big)_{(\ell-1) \ee_1+ \bb},
\end{array}
\end{align*}
}

\noindent
where $\Theta'=(\theta_i~:~\deg \theta_i = \ee_1)$ and $\varphi $ is the surjection guaranteed by \Cref{3.1}.

Since $\lk_\Delta(\sigma)$ has the dual WLP over $\FF$ for all $\sigma\in\Delta$ with $\pi(\sigma)=\bb$, we conclude from Lemma \ref{dualwlp} that 
for a generic choice of $\Theta'$ and a generic linear form $\omega$ with $\deg \omega = \ee_1$,
the left vertical map is surjective.
Hence, the multiplication map
$$\cdot \omega : \big(\cano{\FF[\Delta]}/\Theta \cano{\FF[\Delta]}\big)_{(\ell-1) \ee_1 + \bb} \to \big(\cano{\FF[\Delta]}/\Theta \cano{\FF[\Delta]}\big)_{\ell \ee_1 + \bb}$$
is also surjective. Furthermore,
since $\mideal \cdot\Sigma(\Theta;\cano{\FF[\Delta]})/\Theta \cano{\FF[\Delta]})$ is zero by Theorem \ref{2.4},
the above surjection gives rise to a surjection
\begin{equation}\label{eq:4.1}
\cdot \omega : \big(\cano{\FF[\Delta]}/\Sigma(\Theta; \cano{\FF[\Delta]})\big)_{(\ell-1)\ee_1 + \bb} \to \big(\cano{\FF[\Delta]}/\Theta \cano{\FF[\Delta]}\big)_{\ell \ee_1 + \bb}.
\end{equation}
Therefore,
\begin{equation}\label{eq:dim}
\dim_\FF \big(\cano{\FF[\Delta]}/\Sigma(\Theta; \cano{\FF[\Delta]}) \big)_{(\ell-1)\ee_1 + \bb}\geq \dim_\FF \big(\cano{\FF[\Delta]}/\Theta \cano{\FF[\Delta]}\big)_{\ell \ee_1 + \bb}.
\end{equation}

The right-hand-side of \eqref{eq:dim} can be rewritten as
\begin{align*}
&\dim_\FF \big(\cano{\FF[\Delta]}/\Theta \cano{\FF[\Delta]}\big)_{\ell \ee_1 + \bb}\\
=& \dim_\FF \big(\cano{\FF[\Delta]}/\Sigma(\Theta;\cano{\FF[\Delta]})\big)_{\ell \ee_1 + \bb}+\dim_\FF \big(\Sigma(\Theta;\cano{\FF[\Delta]})/\Theta\cano{\FF[\Delta]}\big)_{\ell \ee_1 + \bb}\\
=&\dim_\FF \big(\cano{\FF[\Delta]}/\Sigma(\Theta;\cano{\FF[\Delta]})\big)_{\ell \ee_1 + \bb}+\binom{2 \ell-1}{\ell}\tilde{\beta}_{\ell-1}(\Delta),
\end{align*}
where the last equality follows from Theorems \ref{2.4} and \ref{thm:cano2}. 
In addition, it follows from Theorems \ref{2.6} and \ref{dual} that for all $\bb \leq \aaa$,
$$\dim_\FF \big(\cano{\FF[\Delta]}/\Sigma(\Theta;\cano{\FF[\Delta]})\big)_{\aaa-\bb} =
\dim_\FF \big({\FF[\Delta]}/\Sigma(\Theta;{\FF[\Delta]})\big)_{\bb}=h''_\bb(\Delta,\pi).$$
Substituting these formulas in \eqref{eq:dim}, we infer that 
\begin{equation*}
h''_{\ell \ee_1}(\Delta,\pi)\geq h''_{(\ell-1)\ee_1}(\Delta,\pi)+\binom{2\ell -1 }{\ell}\tilde{\beta}_{\ell-1}(\Delta),
\end{equation*}
as desired.
\end{proof}

\Cref{3.3} implies the following theorem exactly in the same way as \Cref{numerics} implied \Cref{thm:GLBT last step}. 

\begin{theorem}\label{3.4}
Let $1 < \ell \leq \lfloor d/2 \rfloor$.
Let $(\Delta,\pi)$ be a balanced Buchsbaum simplicial complex of dimension $d-1$.
Suppose that for all faces $\sigma\in \Delta$ of codimension $2\ell-1$, the link of $\sigma$ has the dual WLP. Then
\begin{itemize}
\item[(i)] for any $S\subseteq [d]$ with $|S|=2\ell -1$, 
\[\sum_{T\subseteq S,|T|=\ell}h''_T(\Delta,\pi)-\sum_{T\subseteq S,|T|=\ell-1}h''_T(\Delta,\pi)\geq {2\ell-1\choose \ell}\tilde{\beta}_{\ell-1}(\Delta).\]
\item[(ii)] Consequently,  $\hav''_\ell(\Delta)-\hav''_{\ell-1}(\Delta)\geq \tilde{\beta}_{\ell-1}(\Delta).$
\end{itemize}
\end{theorem}

\begin{remark}
\label{equalitycase} The inequality in part (ii) of Theorem \ref{3.4} (\Cref{thm:GLBT last step}, resp.) 
holds as equality if and only if the inequality in part (i) holds as equality for \textbf{all} $S \subseteq [d]$ with $|S|=2\ell-1$. 
\end{remark}

Since all proper links of a closed $\FF$-homology manifold are $\FF$-homology spheres, the above theorem implies our \Cref{BGLBT}(i).
Moreover, since all homology $2$-spheres have the (dual) WLP, the inequality part of \Cref{g2''} also follows:

\begin{corollary}
Let $\Delta$ be a balanced $\FF$-homology manifold without boundary.
If $\Delta$ has dimension $\geq 3$, then $\hav''_2(\Delta)-\hav_1''(\Delta) \geq \tilde \beta_1(\Delta)$.
\end{corollary}

\section{The equality part of \Cref{g2''}}\label{sect:Equality}

In this section, we complete the proof of \Cref{g2''}, see \Cref{4.5}.
We first recall some results on stacked cross-polytopal spheres verified in \cite{KN}.

Let $\Delta$ and $\Gamma$ be pure simplicial complexes of the same dimension with disjoint vertex sets.
Let $\sigma \in \Delta$ and $\tau \in \Gamma$ be facets and let $\varphi: \sigma \to \tau$ be a bijection.
The \textbf{connected sum} $\Delta\#_\varphi \Gamma$ of $\Delta$ and $\Gamma$ is the simplicial complex obtained from $(\Delta \setminus\{\sigma\}) \cup (\Gamma \setminus \{\tau\})$ by identifying $v$ with $\varphi (v)$ for all $v \in \sigma$.
If $\Delta$ and $\Gamma$ are balanced, then so is $\Delta \#_\varphi \Gamma$. 
A \textbf{stacked cross-polytopal sphere} of dimension $d-1$ is the connected sum of several copies of the boundary complex of the $d$-dimensional cross-polytope. 

The following result was established in \cite[Theorem 4.1 and Lemma 4.2]{KN}. Below, $\hav_i(\Delta):=\frac{h_i(\Delta)}{{d \choose i}}$ denotes the \textbf{normalized $h_i$-number} of $\Delta$.

\begin{theorem}[Klee--Novik]
\label{4.1}
Let $(\Delta,\pi)$ be a  balanced $\FF$-homology manifold of dimension $d-1 \geq 3$.
Then the following conditions are equivalent:
\begin{itemize}
\item[(i)] $\Delta$ is a stacked cross-polytopal sphere.
\item[(ii)] $\hav_2(\Delta)-\hav_1(\Delta)=0$.
\item[(iii)] For any $S \subseteq [d]$ with $|S|=3$,
$\sum_{T \subseteq S, |T|=2} h_T(\Delta,\pi)
=\sum_{T \subseteq S, |T|=1} h_T(\Delta,\pi)$.
\end{itemize}
\end{theorem}

Let $(\Delta,\pi)$ be a pure balanced simplicial complex.
Let $\sigma$ and $\tau$ be facets of $\Delta$ and let $\varphi: \sigma \to \tau$ be a bijection with $\pi(v)=\pi(\varphi(v))$ for all $v \in \sigma$. Such a bijection
 $\varphi$ is called \textbf{admissible} if $\lk_\Delta(v) \cap \lk_\Delta(\varphi(v))= \{\emptyset\}$ for all $v \in \sigma$.
For an admissible bijection $\varphi$, define $\Delta^\varphi$ as the simplicial complex obtained from $\Delta \setminus \{\sigma, \tau\}$ by identifying $v$ with $\varphi(v)$ for all $v \in \sigma$.
We say that $\Delta^\varphi$ is obtained from $\Delta$ by a \textbf{balanced handle addition}.
The \textbf{balanced Walkup class} $\mathcal{BH}^d$ is the set of all balanced simplicial complexes obtained from the boundary complexes of $d$-dimensional cross-polytopes by successively applying the operations of connected sums and balanced handle additions.

The following result is \cite[Corollary 4.12]{KN}.

\begin{theorem}[Klee--Novik]
\label{walkupclass}
Let $\Delta$ be a balanced $\FF$-homology manifold of dimension $d-1 \geq 4$.
Then $\Delta \in \mathcal {BH}^d$ if and only if all vertex links of $\Delta$ are stacked cross-polytopal spheres.
\end{theorem}



Our proof of the equality part of \Cref{g2''} relies on the following lemma.

\begin{lemma}
\label{4.3}
Let $\Delta$ be a $(d-1)$-dimensional connected $\FF$-homology manifold without boundary
and let $\Theta$ be an l.s.o.p.\ for $\FF[\Delta]$. If $d \geq 2$,
then for any vertex $v$ of $\Delta$ with $x_v \not \in (\Theta)$,
there is an injection
\begin{align*}
\varphi : \cano{\FF[\st_\Delta(v)]}/\Theta \cano{\FF[\st_\Delta(v)]} \to \cano{\FF[\Delta]} / \Theta \cano{\FF[\Delta]}
\end{align*}
such that its composition with the natural surjection
from $\cano{\FF[\Delta]} /\Theta \cano{\FF[\Delta]}$
to
$\cano{\FF[\Delta]} /\Sigma(\Theta; \cano{\FF[\Delta]})$
is also an injection.
\end{lemma}

\begin{proof}
We assume that $V(\Delta)=\{1,2,\dots,n\}$ and $A=\FF[x_1,\dots,x_n]$.
As shown in the proof of Lemma \ref{3.0} (see \eqref{3-1-1}), there is an injection
\begin{align}
\label{7.8.20}
\cano{\FF[\st_\Delta(v)]} \to \cano{\FF[\Delta]}.
\end{align}
This injection gives rise to  an $A$-homomorphism 
$$\varphi:
\cano{\FF[\st_\Delta(v)]}/\Theta \cano{\FF[\st_\Delta(v)]} \to \cano{\FF[\Delta]}/ \Theta \cano{\FF[\Delta]}.$$ 
Composing this $A$-homomorphism with the natural surjection 
$\cano{\FF[\Delta]} /\Theta \cano{\FF[\Delta]} \to  \cano{\FF[\Delta]} /\Sigma(\Theta; \cano{\FF[\Delta]})$,
leads to an $A$-homomorphism
\begin{align}
\label{4-1}
\varphi' : \cano{\FF[\st_\Delta(v)]}/\Theta \cano{\FF[\st_\Delta(v)]} \to \cano{\FF[\Delta]} /\Sigma(\Theta; \cano{\FF[\Delta]}).
\end{align}
Thus to prove the desired statement,
it suffices to show that the map $\varphi'$ is injective.

Note that $\Theta$ is an l.s.o.p.\ for $\FF[\st_\Delta(v)]$ since $\st_\Delta(v)$ is a full-dimensional subcomplex of $\Delta$. 
Taking the Matlis dual of modules in \eqref{4-1} and using \Cref{dual},
leads to an $A$-homomorphism
\begin{align*}
\FF[\Delta]/\Sigma(\Theta;\FF[\Delta]) \to \FF[\st_\Delta(v)]/\Theta \FF[\st_\Delta(v)].
\end{align*}
Since for all graded ideals $I$ and $J$ of $A$,
any $A$-homomorphism from $A/I$ to $A/J$ of degree $0$ must be either zero or surjective,
the map $\varphi'$ is either zero or injective.
We prove that $\varphi'$ is non-zero.

Since $\st_\Delta(v)$ is a cone over an $\FF$-homology sphere $\lk_\Delta(v)$, it follows that
$\FF[\st_\Delta(v)]$ is a Gorenstein ring that is isomorphic to $\cano{\FF[\st_\Delta(v)]}(+\eenu_v)$; in particular, $\FF[\st_\Delta(v)]_0\cong \cano{\FF[\st_\Delta(v)]}_{\eenu_v}$.
Let $\alpha$ be a non-zero element of $\cano{\FF[\st_\Delta(v)]}_{\eenu_v}$.
We claim that $\varphi'(\alpha)$ is non-zero.
Since $H_\mideal^i(\cano{\FF[\Delta]})=0$ for $i \leq 1$
(see \cite[Lemma 1]{Ao80}),
Theorems \ref{2.4} and \ref{thm:cano2} imply that
$$\big( \cano{\FF[\Delta]}/ \Sigma(\Theta;\cano{\FF[\Delta]}) \big)_1 =\big( \cano{\FF[\Delta]}/\Theta \cano{\FF[\Delta]} \big)_1.$$
Thus, to prove that $\varphi'(\alpha) \ne 0$,
it is enough to show that $\varphi (\alpha) \ne 0$.

If $\Delta$ is orientable, then $\FF[\Delta]$ is isomorphic to $\cano{\FF[\Delta]}$ by a result of Gr\"abe \cite{Gr}.
Since the map in \eqref{7.8.20} preserves the $\NN^{n}$-grading and $\alpha$ has degree $\eenu_v$, $\varphi(\alpha) \in \cano{\FF[\Delta]}/\Theta \cano{\FF[\Delta]}$ can be identified with  a non-zero scalar multiple of $x_v$ in $\FF[\Delta]/\Theta \FF[\Delta]$, which is a non-zero element by our assumption that $x_v \not \in (\Theta)$.
Suppose that $\Delta$ is non-orientable.
Then $\cano{\FF[\Delta]}_0 \cong \tilde H_{d-1}(\Delta;\FF)$ is zero by Hochster's formula,
and so
$$\big( \cano{\FF[\Delta]}/\Theta \cano{\FF[\Delta]} \big)_1=\cano{\FF[\Delta]}_1.$$
In this case, the fact that $\varphi(\alpha)$ is non-zero follows from the injectivity of \eqref{7.8.20}.
\end{proof}

We now turn to the proof of the main result of this section
which completes the proof of \Cref{g2''}.

\begin{theorem}
\label{4.5}
Let $(\Delta,\pi)$ be a balanced, connected, $\FF$-homology manifold without boundary of dimension $d-1 \geq 4$.
Then $\hav''_2(\Delta)-\hav''_1(\Delta)= \tilde{\beta}_1(\Delta)$ if and only if $\Delta \in \mathcal {BH}^d$.
\end{theorem}

\begin{proof}
As noted at the end of Section 4 in \cite{KN}, the ``if"-part is easy.
We prove the ``only if"-part.
Let $S \subseteq [d]$ with $|S|=3$ and let $v$ be a vertex of $\Delta$ with $\pi(v) \not \in \{ \ee_i~:~ i \in S\}$.
By Theorems \ref{4.1} and \ref{walkupclass}, it suffices to  check that
\begin{align}
\label{4-4}
\sum_{T \subseteq S,\ |T|=2} h_T(\lk_\Delta(v),\pi)=
\sum_{T \subseteq S,\ |T|=1} h_T(\lk_\Delta(v),\pi).
\end{align}

We may assume that $S=\{1,d-1,d\}$.
Let $\aaa=(3,1,\dots,1) \in \NN^{d-2}$.
Define $\tilde \pi : V(\Delta) \to \{\ee_1,\dots,\ee_{d-2}\}$ by $\tilde \pi(u) = \pi(u)$ if $\pi (u) \not \in \{ \ee_i: i \in S\}$ and $\tilde \pi(u)=\ee_1$ if $\pi(u) \in \{\ee_i:i \in S\}$.
Then $(\Delta,\tilde \pi)$ is $\aaa$-balanced, 
$h''_{\ee_1}(\Delta,\tilde \pi)=\sum_{T \subseteq S, |S|=1} h''_T(\Delta,\pi)$, and
$h_{2 \ee_1}''(\Delta,\tilde \pi)=\sum_{T \subseteq S, |S|=2} h''_T(\Delta,\pi)$.
The proof of \Cref{3.3} (see \eqref{eq:4.1}) shows that there is an $\NN^m$-graded l.s.o.p.\ $\Theta$ for $\FF[\Delta]$ and a linear form $\omega$ with $\deg \omega = \ee_1$ such that
$$\cdot \omega : \big(\cano{\FF[\Delta]}/\Sigma(\Theta;\cano{\FF[\Delta]})\big)_{\aaa-2\ee_1}
\to \big(\cano{\FF[\Delta]}/\Theta\cano{\FF[\Delta]}\big)_{\aaa-\ee_1}$$
is surjective. 
Since, by our assumption, $\hav_2''(\Delta)-\hav_1''(\Delta)= \tilde \beta_1(\Delta)$, it follows from \Cref{equalitycase} that
 $h_{2 \ee_1}''(\Delta,\tilde \pi)- h_{\ee_1}''(\Delta,\tilde \pi) = 3 \tilde \beta_1(\Delta)$.
The proof of \Cref{3.3} then implies  that the above map is, in fact, an isomorphism. 

We have the following commutative diagram
\begin{align*}
\begin{array}{ccc}
\big(\cano{\FF[\st_\Delta(v)]}/\Theta \cano{\FF[\st_\Delta(v)]}\big)_{\aaa-\ee_1}
& \stackrel {\varphi'} \longrightarrow & 
\big(\cano{\FF[\Delta]}/\Theta \cano{\FF[\Delta]}\big)_{\aaa-\ee_1}\\
\uparrow \cdot \omega & & \uparrow \cdot \omega\\
\big(\cano{\FF[\st_\Delta(v)]}/\Theta\cano{\FF[\st_\Delta(v)]}\big)_{\aaa-2\ee_1}
& \stackrel {\varphi} \longrightarrow & 
\big(\cano{\FF[\Delta]}/\Sigma(\Theta; \cano{\FF[\Delta])}\big)_{\aaa-2\ee_1},
\end{array}
\end{align*}
where $\varphi$ and $\varphi'$ are injections given in Lemma \ref{4.3}.
Since the right vertical map and the lower horizontal map are injective,
we conclude that the left vertical map is injective.
This implies that
$$\dim_\FF ( \cano{\FF[\st_\Delta(v)]}/\Theta \cano{\FF[\st_\Delta(v)]})_{\aaa-\ee_1}
\geq
\dim_\FF ( \cano{\FF[\st_\Delta(v)]}/\Theta \cano{\FF[\st_\Delta(v)]})_{\aaa-2\ee_1},
$$
and, since the star, $\st_\Delta(v)$, is Cohen--Macaulay, we infer from Remark \ref{dualCM} that
\begin{align*}
\dim_\FF \big( \FF[\st_\Delta(v)]/\Theta \FF[\st_\Delta(v)]\big)_{\ee_1}
\geq  
\dim_\FF \big( \FF[\st_\Delta(v)]/\Theta \FF[\st_\Delta(v)]\big)_{2\ee_1}.
\end{align*}

As the flag $h$-vectors of $\st_\Delta(v)$ and $\lk_\Delta(v)$ coincide, 
the above inequality shows that $h_{\ee_1}(\lk_\Delta(v),\tilde \pi) \geq h_{2 \ee_1} (\lk_\Delta(v),\tilde \pi)$. 
On the other hand, 
since $\tilde \pi(v) \ne \ee_1$ and $d-1\geq 4$, it follows that the link, $\lk_\Delta(v)$, satisfies the assumptions of \Cref{3.3} for $\ell=2$. Hence, by this proposition, $h_{\ee_1}(\lk_\Delta(v),\tilde \pi)\leq h_{2 \ee_1} (\lk_\Delta(v),\tilde \pi)$. We conclude that $h_{\ee_1}(\lk_\Delta(v),\tilde \pi) = h_{2 \ee_1} (\lk_\Delta(v),\tilde \pi)$. The result follows, since according to the definition of $\tilde \pi$, this equality is equivalent to the desired statement \eqref{4-4}.
\end{proof}

\begin{remark} \label{(i-1)stacked}
Let $\ell \leq \lfloor (d-1)/2\rfloor$ and let $\Delta$ be a balanced $(d-1)$-dimensional $\QQ$-homology manifold such that all vertex links of $\Delta$ are polytopal. In the same way as in the proof of \Cref{4.5}, one can show that if 
$\hav_\ell''(\Delta)-\hav_{\ell-1}''(\Delta)=\tilde \beta_{\ell-1}(\Delta)$, then $\hav_\ell(\lk_\Delta(v))=\hav_{\ell-1}(\lk_\Delta(v))$ for all $v \in V(\Delta)$.
Along with a recent result of Adiprasito \cite[Section 9]{Adip}, this, in turn, implies that all vertex links of $\Delta$ have the balanced $(\ell-1)$-stacked property. (A balanced analog of the $(\ell-1)$-stacked property for homology spheres and manifolds was defined in \cite[Definition 5.3]{KN}.)
\end{remark}

\section{Closing remarks and open problems}
We close with several remarks as well as some problems related to this paper that we left unsolved. 

\subsection{Buchsbaum* simplicial complexes}
Our proof for the non-orientable case (see \Cref{3.4}) applies not only to homology manifolds but also to Buchsbaum* complexes.
A simplicial complex $\Delta$ of dimension $d-1$ is \textbf{Buchsbaum*} (over $\FF$) if it is Buchsbaum (over $\FF$) and, in addition,   $\tilde H_{d-2}(|\Delta|-p;\FF)\cong \tilde H_{d-2}(|\Delta|;\FF)$ for every point $p$ in the geometric realization $|\Delta|$ of $\Delta$. A simplicial complex $\Delta$ of dimension $d-1$ is  called \textbf{doubly Cohen--Macaulay} 
(over $\FF$) if it is Cohen--Macaulay, and, in addition, $\Delta\setminus v=\{F\in \Delta~:~v\notin F\}$ is Cohen--Macaulay of dimension $d-1$ for every vertex $v \in V(\Delta)$.

It was shown in \cite[Theorem 9.8]{Walker} that being doubly Cohen--Macaulay is equivalent to being both Buchsbaum* and Cohen--Macaulay, while according to
\cite[Corollary 2.12]{AW-12}, every proper link of a Buchsbaum* simplicial complex is doubly Cohen--Macaulay. Furthermore,  Bj\"orner and Swartz suggested
the following conjecture, see \cite[Problem 4.2]{Swartz-06}.

\begin{conjecture}[{Bj\"orner--Swartz}] All doubly Cohen--Macaulay simplicial complexes have the dual WLP.
\label{BS}
\end{conjecture}

\noindent Thus, if true, \Cref{BS} would imply that the conclusions of Theorem \ref{3.4} hold for all Buchsbaum* simplicial complexes. Recall that \Cref{BS} does hold in dimension two. (Indeed, since by \cite{Nevo} all $2$-dimensional doubly Cohen--Macaulay complexes are minimal cycles in the sense of \cite{Fo}, the statement in characteristic zero follows from the main result of \cite{Fo}. For nonzero characteristic, see the discussion and references in \cite[Section 5]{NS3}.) Hence we obtain the following result that strengthens the result of  Browder and Klee \cite[Theorem 4.1]{BK}.

\begin{theorem}
Let $\Delta$ be a balanced Buchsbaum* simplicial complex of dimension $\geq 3$, then $\hav''_2(\Delta) -\hav''_1(\Delta) \geq \tilde \beta_1(\Delta)$.
\end{theorem}

In \cite[Question 5.7(ii)]{AW-12}, Athanasiadis and Welker asked whether for a $(d-1)$-dimensional Buchsbaum* simplicial complex, the vector given by the successive differences of the first half of the $h''$-vector is an $M$-sequence (that is, the Hilbert function of some standard graded $\FF$-algebra). 
Our next result shows that the validity of Conjecture \ref{BS} would provide an affirmative answer to this question.

\begin{proposition}
Let $\Delta$ be a connected Buchsbaum* simplicial complex of dimension $d-1$.
If all vertex links of $\Delta$ have the dual WLP, then the vector 
\begin{align}
\label{5-1}
\textstyle (h_0''(\Delta),h_1''(\Delta)-h_0''(\Delta),\dots,h''_{\lfloor d /2 \rfloor}(\Delta)-h''_{\lfloor d/ 2 \rfloor-1}(\Delta))
\end{align}
is an $M$-sequence.
\end{proposition}

\begin{proof}(Sketch)
We start with two observations. First, by Lemma \ref{3.0}, there is a surjection
$$
N:=\bigoplus_{v \in V(\Delta)} \cano {\FF[\st_\Delta(v)]}/\Theta \cano {\FF[\st_\Delta(v)]}
\to \bigoplus_{k \geq 1} \big(\cano {\FF[\Delta]}/\Sigma(\Theta; \cano {\FF[\Delta]} \big)_k,$$
where $\Theta$ is a generic l.s.o.p. Second,
since $x_v$ is a non-zero divisor on $\FF[\st_\Delta(v)]$ and $\FF[\st_\Delta(v)]/(x_v)\FF[\st_\Delta(v)]\cong\FF[\lk_\Delta(v)]$ while $\lk_\Delta(v)$ is a $(d-2)$-dimensional complex that has the dual WLP, it follows that
for a generic linear form $\omega$, the multiplication map $\cdot \omega: N_k \to N_{k+1}$ is surjective for $k = \lfloor (d-1)/ 2\rfloor$. These two observations imply that multiplication by $\omega$ on $\cano{\FF[\Delta]}/\Sigma(\Theta; \cano {\FF[\Delta]})$ from degree $\lfloor (d-1)/2 \rfloor$ to degree $\lfloor (d-1)/2 \rfloor+1$ is also surjective.
Thus, by Theorem \ref{dual}, multiplication by $\omega$ on $\FF[\Delta]/\Sigma(\Theta; \FF[\Delta])$ from degree $\lfloor d/2\rfloor -1$ to degree $\lfloor d/2\rfloor$ is injective. Then, since $\FF[\Delta]/\Sigma(\Theta; \FF[\Delta])$ is a level algebra by \cite{Nagel}, it follows from Proposition 2.1(b) in \cite{MMN} that the map $\cdot \omega: (\FF[\Delta]/\Sigma(\Theta; \FF[\Delta]))_k\to (\FF[\Delta]/\Sigma(\Theta; \FF[\Delta]))_{k+1}$ is injective for all $k\le \lfloor d/2 \rfloor -1$, and so the vector in (\ref{5-1}) is the Hilbert function of $\FF[\Delta]/(\Sigma(\Theta;\FF[\Delta])+\omega \FF[\Delta])$.
\end{proof}

\subsection{Open problems}
Theorem \ref{g2''} is stated for the class of $\FF$-homology manifolds. However, an analogous ``non-balanced" result is known to hold in a larger generality: the inequality $h''_2(\Delta)\geq h''_{1}(\Delta)+d\tilde{\beta}_{1}(\Delta)$ (for $d\geq 4$) is proved in \cite[Theorem 5.3]{Mu} for \emph{all} normal pseudomanifolds. Thus, it is tempting to conjecture that the statement of \Cref{g2''} remains valid for all balanced normal pseudomanifolds.

Bagchi and Datta \cite{Bagchi-Datta} introduced the notion of $\mu$-numbers. These numbers satisfy the following Morse-type inequalities: for any simplicial complex $\Delta$,  
\[
\sum_{k=0}^j(-1)^{j-k}\mu_k(\Delta)\geq \sum_{k=0}^j(-1)^{j-k}\tilde{\beta}_k(\Delta) + (-1)^j \quad \mbox{for all } j\geq 0.
\]
  In light of \cite[Theorem 6.5]{MN-bdry} and \cite[Theorem 7.3]{MN-bdry}, we conjecture that the statements of Theorem \ref{BGLBT} and Corollary \ref{g2''} can be strengthened by replacing the Betti numbers with the $\mu$-numbers. Furthermore, we conjecture that the resulting inequality in fact holds for all (not necessary orientable) homology manifolds with boundary if we replace $\Delta$ with the pair $(\Delta,\partial \Delta)$. 

We now turn our discussion towards characterizing the cases of equality in \Cref{BGLBT}(i) and \Cref{g2''}.  \Cref{g2''} provides such a characterization when $d\geq 5$ but leaves the case of $d=4$ open. However, in view of \cite[Theorem 5.3]{Mu},  it is plausible that the same characterization continues to hold in the $d=4$ case. As for Theorem \ref{BGLBT}(i), Remark \ref{(i-1)stacked} along with \cite[Theorem 4.6 \& Corollary 5.8]{MN-14} leads us to conjecture  that if $\Delta$ satisfies the assumptions of Theorem \ref{BGLBT}(i) and $\ell<d/2$, then 
\[\hav''_{\ell}(\Delta) = \hav''_{\ell-1}(\Delta) + \tilde{\beta}_{\ell-1}(\Delta)
\]
if and only if $\Delta$ has the balanced $(\ell-1)$-stacked property.

\bibliographystyle{alpha}
\bibliography{BLBT}

\end{document}